\documentclass[11pt]{amsart}
\usepackage{amsfonts}
\usepackage{amsmath}
\usepackage{amsthm}
\usepackage{mathrsfs}
\usepackage{bm,bbm}
\usepackage[dvips]{graphicx}
\usepackage{caption}
\usepackage{natbib}
\usepackage{color}
\usepackage{commath}

\usepackage{bigints}

\usepackage{comment}

\numberwithin{equation}{section}

\allowdisplaybreaks

\theoremstyle{plain}
\newtheorem{theorem}{Theorem}[section]
\newtheorem{lemma}[theorem]{Lemma}
\newtheorem{proposition}[theorem]{Proposition}

\theoremstyle{definition}
\newtheorem{definition}[theorem]{Definition}
\newtheorem{remark}[theorem]{Remark}

\def\beqn{\begin{equation}}
\def\beqn*{$$}
\def\eeqn{\end{equation}}

\def\ms{\mathsf}

\newcommand{\BX}{{\bf X}}
\newcommand{\BY}{{\bf Y}}

\def\P{\mathbb{P}}
\def\E{\mathbb{E}}

\newcommand{\reals}{{\mathbb R}}
\newcommand{\bbr}{\reals}
\newcommand{\R}{\reals}
\newcommand{\bbn}{{\mathbb N}}

\newcommand{\vep}{\varepsilon}

\newcommand{\eid}{\stackrel{d}{=}}

\newcommand{\ta}{\theta}

\newcommand{\one}{{\mathbbm 1}}

\newcommand{\remove}[1]{}

\newcommand{\G}{\mathcal G}

\newcommand{\mH}{\mathcal H}
\newcommand{\mK}{\mathcal K}

\newcommand{\mF}{\mathcal F}
\newcommand{\mG}{\mathcal G}

\newcommand{\mD}{\mathcal D}

\newcommand{\jointp}{p_{i_0,\dots,i_k}}

\newcommand{\inci}{1\le i_k < \dots < i_0}

\newcommand{\marp}{\widetilde p^{(0)}}
\newcommand{\marm}{\widetilde m^{(0)}}
\newcommand{\marvep}{\widetilde \vep^{(0)}}

\begin{document}

 \bibliographystyle{abbrvnat}

\renewcommand{\baselinestretch}{1.05}

\title[Degree counts in  preferential attachment complexes]
{Degree counts in random
  simplicial complexes of the preferential attachment type}

\author{Takashi Owada}
\address{Department of Statistics\\
Purdue University \\
IN, 47907, USA}
\email{owada@purdue.edu}

\author{Gennady Samorodnitsky}
\address{School of Operations Research and Information Engineering\\
Cornell University \\
NY, 14853, USA}
\email{gs18@cornell.edu}

\thanks{Owada's research was partially supported by the AFOSR grant
  FA9550-22-1-0238 at Purdue University. Samorodnitsky's research was partially supported by the AFOSR grant
  FA9550-22-1-0091 and NSF grant 2310974 at Cornell University.}

\subjclass[2010]{Primary 60G70, 05C80.  Secondary 60B12, 60F15, 60J27, 55U10}
\keywords{topological preferential attachment model; simplicial complex; regular variation; power-law tail; degree distribution.}

\begin{abstract}
 We extend  the classical preferential attachment random graph model to
random simplicial complexes. 
At each stage of the model, we choose one of the existing
$k$-simplices with probability proportional to its $k$-degree. The
chosen $k$-simplex then forms a $(k+1)$-simplex with a newly arriving
vertex. We  establish a strong  law of large numbers for the degree
counts across multiple dimensions.  The limiting probability mass
function is expressed as a mixture of  mass functions of different
types of negative binomial random variables. This limiting
distribution has  power-law characteristics and we explore the
limiting extremal dependence  of the degree counts across different
dimensions in the framework of  multivariate regular
variation. Finally, we prove multivariate  weak convergence, under appropriate
normalization, of degree counts in different  dimensions, of ordered
$k$-simplices. The resulting weak limit can  be represented as a
function of independent linear birth  processes with immigration.
\end{abstract}

\maketitle

\section{Introduction}

\subsection{Preferential attachment model}

 We extend the classical preferential attachment random graph model to
random simplicial complexes, which can be viewed as higher-dimensional
objects with an interesting topological structure. 
 Initially formulated to explain the "rich-get-richer" phenomena, the
 preferential attachment model has been applied  in the domain of
 graph theory, where this phenomenon is characterized by vertex
 degrees. Specifically, vertices are sequentially added and connected
 to edges in each stage of the model. The probability of  connection
 between a newly arriving vertex and existing ones is proportional to
 the degrees of the existing vertices. Consequently, vertices with
 already higher degrees have a stronger tendency of attracting new
 vertices compared to those with lower degrees.

The preferential attachment graph model is characterized by its
so-called {\it scale-free property}, as expressed in a power-law
distribution of the vertex degrees. Assuming a network of $n$
vertices, let $p_i(n)$ denote the  proportion of vertices with degree
$i$.  The power-law behavior can be  expressed as 
\begin{equation}  \label{e:power-law.intro}
\lim_{n\to\infty} p_i(n) \approx \alpha i^{-\beta},  \ \ \ i\to\infty, 
\end{equation}
where $\alpha, \beta$ are positive parameters (assuming that the
limits exist, as they often do).  The less rigorous interpretation of
 \eqref{e:power-law.intro} says that, for large $n$, $p_i(n) \approx
 \alpha i^{-\beta}$ for a range of large values of $i$. In this
 interpretation, the exponent  
$\beta$ remains independent of the  network's size for large $n$. The
preferential attachment model for such  evolving networks was
initially proposed by \cite{barabasi:albert:1999} and subsequently
subjected to more rigorous mathematical analysis by
\cite{bollobas:riordan:spencer:tusnady:2001}.  Since then, the
preferential attachment model has received  significant attention and
expanded its scope, both to undirected preferential attachment models
(\cite{athreya:ghosh:sethuraman:2008, berger:borgs:chayes:saberi:2014,
  garavaglia:stegehuis:2019}), and to directed models
(\cite{bollobas:borgs:chayes:riordan:2003,
  samorodnitsky:resnick:towsley:davis:willis:wan:2016,
  wang:resnick:2017, wang:resnick:2020}), among  others. A
comprehensive overview of this  topic is in \cite{hofstad:2017}.

Random simplicial complexes of preferential attachment type have
already received some attention in non-mathematical literature  (e.g.
\cite{courtneya:bianconi:2017, 
  rivkind:schreier:brenner:barak:2020}), but only few rigorous results
have been available.  The paper  by 
\cite{fountoulakis:iyer:mailler:sulzbach:2022} considers a model that
has implicit preferential attachment aspects  and established limit
theorems, including laws of large numbers, for asymptotic degree
counts. Furthermore, \cite{siu:samorodnitsky:yu:he:2024} explored the
clique complex built on
the affine preferential attachment graphs  of 
\cite{garavaglia:stegehuis:2019}, and obtained results on the 
growth rate of expected Betti numbers.   We, on the other hand,
introduce a model that is a direct higher-dimensional extension 
of the preferential attachment graph ideas.

Our model allows us to investigate  degrees in different
dimensions simultaneously. We do  so  in two different ways, starting with an 
investigation of 
asymptotic joint degree distributions across various dimensions and  establishing an explicit construction of the joint limiting
distribution. This limiting distribution turns out to possess  the
property of nonstandard multivariate regular variation, which, in particular, gives
us  different  marginal power laws in different
dimensions. Furthermore,  we 
uncover the dependence structure in the extremes of degrees in
different dimensions, through the explicit computation of the
corresponding tail measure.  Both the joint limiting
distribution of the degree counts and the tail measure involve new (to
the best of our knowledge) classes of distributions. 
 Additionally, we analyze the degree counts in different  dimensions,
 of appropriately ordered $k$-simplices in the network. As in the
 simpler  case of random graphs, these degree counts can be
 represented as functions of independent linear birth  processes with
 immigration,  but in the case of random simplicial complexes,  the obtained 
 representation is much more involved. This representation allows us 
to show that,
 under appropriate normalization, these degree counts have a joint
 weak limit.   

\subsection{Setup}

As outlined above, our preferential attachment model is a random
simplicial complex, so we start by defining what a simplicial complex
is. 

\begin{definition}
An (abstract) simplicial complex $\mK$ is a collection of sets
satisfying the following condition: if $\sigma\in \mK$ and $\tau$ is a
non-empty subset of $\sigma$, then $\tau \in \mK$. The dimension of
$\sigma\in \mK$ is defined as $|\sigma|-1$, where $|\cdot |$ denotes
cardinality. If the dimension of $\sigma$ is $k \ge 0$, it is called a
$k$-simplex. 
\end{definition}
One usually interprets a simplicial  complex in a geometric manner, by
identifying a $0$-simplex with a point (a vertex), a $1$-simplex
with a segment (a usual edge), a $2$-simplex with a filled triangle,
etc, so that the original set is the vertex set of the corresponding  geometric
simplex.

If $\sigma$ is a simplex of dimension $k$ in a simplicial complex
$\mK$, we denote by $D^{(k)}(\sigma)$ the {\it $k$-degree} of $\sigma$ in
$\mK$, defined as the number of the $(k+1)$-simplices in $\mK$
containing $\sigma$ as a subsimplex. Note that the $0$-degree of a
vertex in $\mK$ is simply the number of ($1$-dimensional) edges in
$\mK$ containing this  vertex, i.e. the usual degree of a vertex in a
graph. 

We now present a formal definition of our preferential
attachment model. It is a dynamic random simplicial complex, changing  at discrete times, in a Markovian manner. Fix a parameter
$k\in \{0,1,2,\ldots\}$. 
The model starts  at time 0 from a single initial $(k+1)$-simplex,
defined on a vertex set $[k+2]:= \{ 1,2,\dots,k+2 \}$. We denote the
simplicial complex consisting  of this single $(k+1)$-simplex and its
subsimplices by $G_0$.  At time $1$, a new vertex $k+3$ arrives and
selects one of the $k$-simplices from the 
initial $(k+1)$-simplex (there are $k+1$ of such $k$-simplices).  We
form a new $(k+1)$-simplex by connecting the vertex $k+3$  to each of
the vertices of the selected $k$-simplex, and add this simplex and its
subsimplices to $G_0$, resulting in a new simplicial  complex, denoted
by $G_1$. The vertex set of $G_1$ is $[k+3]$. In general, for $n\ge1$,
the state of the system at time $n$ is a simplicial complex $G_n$ with
a vertex set $[n+k+2]$
 obtained from the simplicial complex $G_{n-1}$ at time $n-1$ as follows. A new 
vertex $n+k+2$ arrives, and selects one the $k$-simplices present in
$G_{n-1}$. We then form a new $(k+1)$-simplex by connecting the vertex
$n+k+2$  to each of 
the vertices of the selected $k$-simplex, and add this simplex and its
subsimplices to $G_{n-1}$, resulting in the simplicial complex $G_n$.

It remains to explain how the newly arriving vertex chooses one of the
existing $k$-simplices. For $m=0,1,\dots,k$, let $\mF_{m,n}$ be the
collection of $m$-simplices in $G_{n}$, and recall that each
$m$-simplex $\sigma$ in $\mF_{m,n}$ has its associated $m$-degree
$D_{n}^{(m)}(\sigma)$ in
$G_n$. The newly arriving  vertex $n+k+2$ at time $n$ chooses a
$k$-simplex in $\mF_{k,n-1}$ with probability proportional to its
$k$-degree plus a parameter $\delta$. That is, given $G_{n-1}$ and
$\sigma\in \mF_{k,n-1}$,  
\begin{equation}  \label{e:selection.prob}
\P(\sigma \text{ is chosen at time } n) =
\frac{D_{n-1}^{(k)}(\sigma)+\delta}{\sum_{\tau \in \mF_{k,n-1}} \big(
  D_{n-1}^{(k)}(\tau ) + \delta\big)}. 
\end{equation}
We emphasize that $\delta \in (-1,\infty)$, along with $k\in
\{0,1,2,\ldots\}$, are parameters for this preferential attachment
process.

Recall that we identify each $m$-simplex $\sigma \in\mF_{m,n}$ with
the set of its vertices, all of which are numbered. The highest-numbered
vertex of a simplex describes the time at which the simplex has been added to
the complex. If the highest-numbered vertex is in the range
$\{1,\ldots, k+2\}$, then the simplex has been added at time 0, and if
the  highest-numbered vertex is $i>k+2$, then the simplex has been
added at time $i-(k+2)$. This allows us to talk meaningfully  about ``older''
simplices and ``newer'' or ``younger" simplices. 

A simple example is presented in Figure \ref{fig:process}. Notice that
if  $\delta$ is large, the selection probability
\eqref{e:selection.prob} is close to a uniform distribution. If
$\delta$ is  smaller, the probability \eqref{e:selection.prob} becomes
more sensitive to the $k$-degree of $\sigma$, and so, the preferential
attachment effect becomes more pronounced when $\delta$ becomes
smaller. 
\begin{figure}
\includegraphics[scale=0.38]{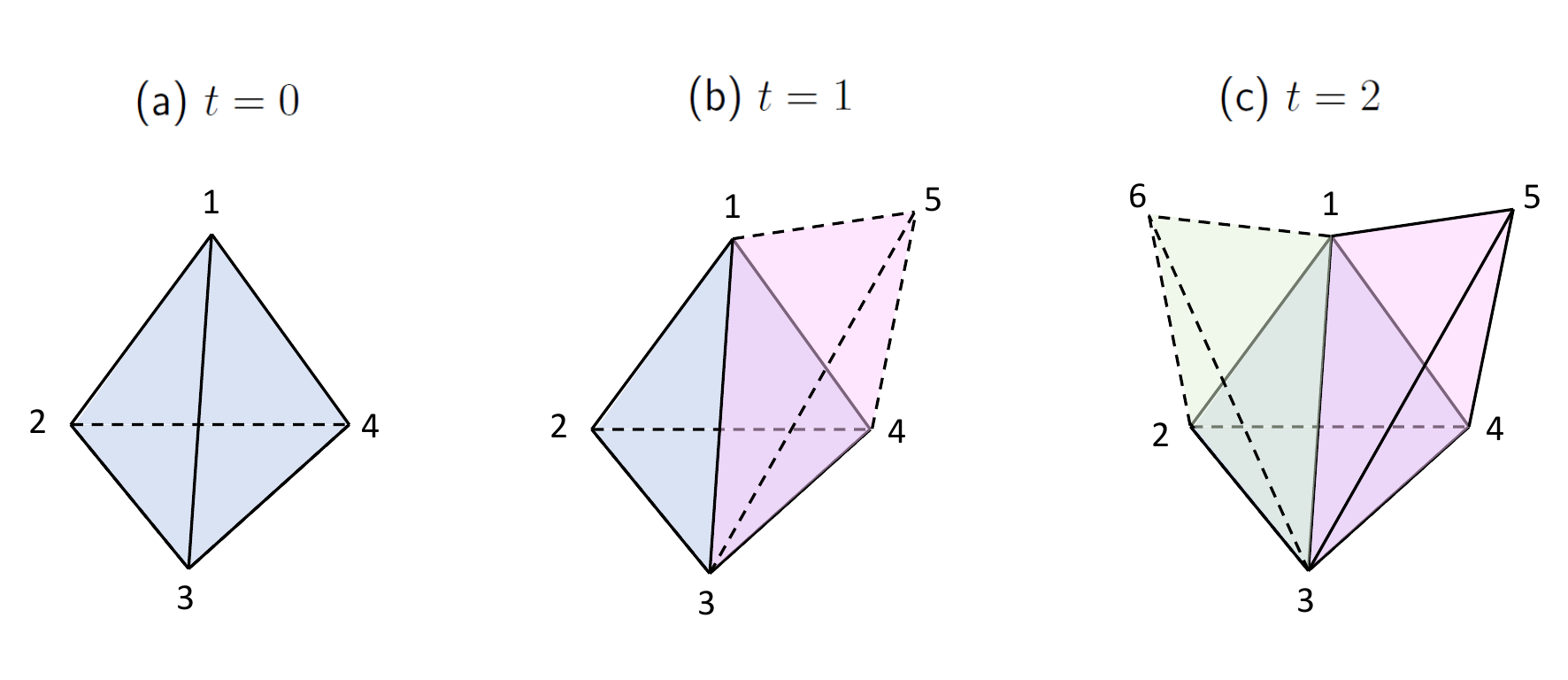}
\caption{\label{fig:process} \footnotesize{Assume that $k=2$ and $\delta=0$. (a) At time $0$, we start with the $3$-simplex on a vertex set $\{ 1,2,3,4 \}$. (b) At time $1$, each of the $2$-simplices $\{ 1,2,3 \}, \{ 1,2,4 \}, \{ 1,3,4 \}$, and $\{ 2,3,4 \}$ is chosen with  equal probability, i.e., $1/4$ each. Suppose $\{ 1,3,4 \}$ is selected at this stage. (c) At time $2$, there are initially seven $2$-simplices: $\{ 1,2,3 \}, \{ 1,2,4 \}, \{ 1,3,4 \}, \{ 2,3,4 \}, \{ 1,3,5 \}, \{ 1,4,5 \}$, and $\{ 3,4,5 \}$. Here, the $2$-degree of $\{ 1,3,4 \}$ is $2$, while all other six $2$-simplices have a $2$-degree $1$. Therefore, the $2$-simplices $\{ 1,3,4 \}$ is chosen with probability $1/4$, whereas all other six $2$-simplices are selected with probability $1/8$.}}
\end{figure}

When a new vertex arrives and selects one of the existing
$k$-simplices, this vertex will add $\binom{k+1}{k}=k+1$ of new
$k$-simplices to the model.  The $k$-degree of the chosen $k$-simplex
increases by $1$, and each of the newly added $k$-simplices has
$k$-degree $1$. Thus, the sum of the $k$-degrees of all $k$-simplices in
the model increases by $k+2$. Observe also that the initial
$(k+1)$-simplex (at time $0$) contains $\binom{k+2}{k+1}=k+2$ of
$k$-simplices, and each of the $k$-simplices has $k$-degree
$1$. Therefore, for every $n\ge0$,  
\begin{equation}  \label{e:number.k-simplices}
\text{the number of } k\text{-simplices in } G_{n} = |\mF_{k,n}| = 1+(n+1)(k+1), 
\end{equation}
and 
\begin{equation}  \label{e:sum.k-degrees.1}
\text{the sum of } k\text{-degrees of all } k\text{-simplices in } G_{n} = (n+1)(k+2). 
\end{equation}
By \eqref{e:number.k-simplices} and \eqref{e:sum.k-degrees.1}, one can rewrite \eqref{e:selection.prob} as
\begin{equation}  \label{e:prob.select.sigma}
\P(\sigma \text{ is chosen at time } n) = \frac{D_{n-1}^{(k)}(\sigma)+\delta}{n(k+2) + \delta(1+n(k+1))}. 
\end{equation}

With every $k$-simplex $\sigma\in \mF_{k,n}$, we can associate a
vector of degrees in dimensions ranging from 0 to $k$. Specifically,
for every $m\in \{  0,\dots,k\}$, we choose the {\it youngest} $m$-simplex 
among the
$k+1\choose m+1$ $m$-simplices contained in $\sigma$. If we denote  $\sigma = \{v_0,v_1,\dots,v_k\}$ with
$v_0< \dots <v_k$, then the $m$-degree of this youngest $m$-simplex is given by $D_n^{(m)}
(v_{k-m},\dots,v_k)$.
The  central object of our study  is
the joint distribution of these  degrees in all dimensions among
the $k$-simplices in $\mF_{k,n}$ as $n\to\infty$. For that purpose, we
define  
\begin{equation}  \label{e:def.N_n}
N_n(i_0,\dots,i_k):= \sum_{\sigma \in \mF_{k,n}} \one \big\{ D_n^{(m)} (v_{k-m},\dots,v_k) = i_m, \, m=0,\dots,k\big\}, 
\end{equation}
where $\sigma = \{v_0,v_1,\dots,v_k\}$ with $v_0< \dots <v_k$, and
$\one \{ \cdot \}$ denotes an indicator function.  
By construction, $N_n(i_0,\dots,i_k)\equiv 0$ whenever $i_m>n+k-m+1$
for some $m\in\{ 0,\dots,k \}$ or $i_m\ge i_{m-1}$ for some $m\in\{ 1,\dots,k \}$.

We organize the rest of this paper  as follows. Section
\ref{sec:regvar} studies the joint empirical probability mass function (pmf)
of the
degrees in all dimensions given by $N_n(i_0,\dots,i_k)/|\mF_{k,n}|$
as a function of $(i_0,\dots,i_k)$. Theorem \ref{t:SLLN} proves that
this empirical  pmf has, with probability 1, a non-random
limit, which itself is a  pmf. One
can think of 
this result as a  strong law of large numbers (SLLN) for the degree
counts in \eqref{e:def.N_n}.  In 
Proposition \ref{p:representation} we show that the limiting
probability distribution can
be represented as a mixture of laws involving certain versions of
independent negative 
binomial random variables. In Theorem \ref{t:regular.variation} we
prove that the limiting joint degree distribution has  a  
power-law behavior in the sense of nonstandard multivariate regular
variation and clarify  the extremal dependence structure  of the
degrees across different dimensions. 
Parts of our argument use a high-dimensional version of techniques
developed in \cite{samorodnitsky:resnick:towsley:davis:willis:wan:2016}.  
In Section \ref{sec:multivariate.weak.conv} we introduce an ordering
of all $k$-simplices ever present in the growing simplicial
complex. Since we associate with each $k$-simplex a vector of degrees
in dimensions 0 through $k$, this ordering gives us, for each time
$n$, a sequence of vectors of degrees. The sequence is finite (since
the sequence of all $k$-simplices is, at time $n$, restricted to the 
simplices in $\mF_{k,n}$), but we view it as a random sequence in
$(\bbn^{k+1})^\infty$ by appending  infinitely
many zeroes to the right of this random sequence. In Theorem \ref{t:joint.weak.conv} we prove that, with
proper scaling, this sequence has a week limit  that can be
represented in terms of certain linear birth processes and 
independent Gamma random variables. In general, some part of our treatment of
preferential attachment simplicial complexes follows the strategy used in \cite{hofstad:2017} for preferential
attachment graphs.

\medskip
\section{Multivariate  regular variation of degree distributions} \label{sec:regvar}

At time $n$, the random simplicial complex $G_n$ has $ |\mF_{k,n}| =
1+(n+1)(k+1)$ of $k$-simplices, so
\begin{equation} \label{e:empirical.pn}
p_{i_0,\dots,i_k}^{(n)}:= \frac{N_n(i_0,\dots,i_k)}{1+(n+1)(k+1)}, \ \ \ 
1\le i_k < \dots <i_0, 
\end{equation}
is a random pmf describing the joint
distribution of the degrees in all dimensions among all $k$-simplices
at that time. Our first result shows that, with probability 1, this
sequence of random pmfs  has a deterministic weak
limit. 

\begin{theorem}  \label{t:SLLN}
For every $1\le i_k<\dots <i_0$, the sequence of random pmfs defined in \eqref{e:empirical.pn} satisfies 
\begin{equation}  \label{e:SLLN}
 p_{i_0,\dots,i_k}^{(n)}\to \jointp \ \ \text{as } n\to\infty, \
 \text{a.s. for each } \, 1\le i_k<\dots <i_0, 
\end{equation}
where $(p_{i_0,\dots,i_k})_{1\le i_k < \dots <i_0}$ is a pmf,   defined recursively by 
\begin{align}
\begin{split}  \label{e:recursive.p}
p_{i_0,\dots,i_k} &= \frac{i_k-1+\delta}{\tau}\, p_{i_0-1,\dots,i_k-1} \\
&\quad +\sum_{m=0}^{k-1} \frac{(i_m-1-k+m)b_m - (i_{m+1}-k+m+1)b_{m+1}}{\tau}\, 
 p_{i_0-1,\dots,i_m-1,i_{m+1},\dots,i_k} \\
&\quad -\frac{(i_0-k)b_0+\delta}{\tau} \, p_{i_0,\dots,i_k} + \one \{ i_\ell=k-\ell+1, \, \ell=0,\dots,k \}, \ \ 1 \le i_k < \cdots < i_0, 
\end{split}
\end{align}
where 
\begin{align} 
\begin{split} \label{e:def.bm.tau}
b_m:= k-m+1 +\delta (k-m), \  \ m=0,\dots,k,   \ \ \text{ and } \ \   \tau:= k+2 +\delta (k+1). 
\end{split}
\end{align}
\end{theorem}

\begin{remark}
One should interpret the probabilities $p_{i_0,\dots,i_k} $  in the right hand
side of the recursion \eqref{e:recursive.p} violating the condition $1
\le i_k < \cdots < i_0$ as equal to 0. It is easy to see that
\eqref{e:recursive.p} is a true recursion in the sense  that it can be solved first
for $(i_0,i_1,\ldots, i_k)=(k+1,k,\ldots, 1)$, then for
$(i_0,i_1,\ldots, i_k)=(k+2,k,\ldots, 1)$, etc. 

\end{remark}

Let
$(D_\infty^{(0)}, \dots, D_\infty^{(k)})$ be a discrete random
vector with the limiting pmf in Theorem \ref{t:SLLN}, i.e., 
\begin{equation}  \label{e:def.D_inf}
\P(D_\infty^{(m)} = i_m, \, m=0,\dots,k) = \jointp, \ \ \inci.
\end{equation}
We now present an explicit construction of this random vector using
simple ingredients. This construction will enable us in the sequel to
uncover  the extremal dependence structure 
of the  limiting degree distributions across different dimensions.

We start by introducing certain random elements that will play a role
in our representation.  With $\tau$ and $b_0$ as  in Theorem
\ref{t:SLLN}, let $Z$ be a Pareto$(\tau/b_0)$ random variable with  density 
\begin{equation}  \label{e:def.Z}
f_Z(z) =\frac{\tau}{b_0}\, z^{-(1+\tau/b_0)}, \ \ \ z\ge 1. 
\end{equation}
We will also use certain parametric families of discrete random variables whose
parameters will be taken from the description of $(\jointp)_{\inci}$ in Theorem
\ref{t:SLLN}. 
Given $r>0$ and $a\in (0,1)$,  the negative binomial
distribution with shape $r$ and probability for success $a$ has a 
pmf $p^{\rm NB}_{r,a}$  given by 
\begin{equation}  \label{e:def.neg.binomial}
p^{\rm NB}_{r,a}(\ell)= \frac{\Gamma(\ell+r)}{\ell!\, \Gamma (r)}\, (1-a)^\ell a^r, \ \ \ \ell\in \bbn:=\{0,1,2,\dots\}, 
\end{equation}
where $\Gamma(\cdot)$ is the Gamma function. Recall that the probability generating function (pgf) of the negative binomial 
distribution satisfies
\begin{equation}  \label{e:pgf.neg.binomial}
\sum_{\ell=0}^\infty x^\ell p^{\rm NB}_{r,a}(\ell) = \big( x(1-a^{-1}) + a^{-1} \big)^{-r}, \ \ \ 0<x<1. 
\end{equation}

Next, the so-called \emph{extended  truncated negative binomial distribution} with parameters 
$\kappa\in (0,1)$ and $a\in (0,1)$ has a pmf $p^{\rm TNB}_{\kappa,a}$  given by 
$$
p^{\rm TNB}_{\kappa,a}(\ell)  = \frac{\kappa\,
  \Gamma(\ell-\kappa)}{\ell!\, \Gamma(1-\kappa)} \cdot
\frac{(1-a)^\ell}{1-a^\kappa}, \ \ \ \ell\in \bbn_+:=\{1,2,\dots\}; 
$$
it can be viewed as a version of the negative binomial distribution
with  a negative shape ``$-\kappa$" in the range $(-1,0)$. It follows
from the binomial theorem that the pgf of the extended
  truncated negative binomial distribution satisfies 
\begin{align}
\begin{split}  \label{e:pgf.ext.trun.neg.binomial}
 \sum_{\ell=1}^\infty x^\ell p^{\rm TNB}_{\kappa,a}(\ell)&= \sum_{\ell=1}^\infty \frac{\kappa\, \Gamma(\ell-k)}{\ell!\, \Gamma(1-\kappa)}\cdot \frac{(1-a)^\ell}{1-a^\kappa}\, x^\ell  =\frac{1}{1-a^\kappa}\sum_{\ell=1}^\infty \binom{\kappa}{\ell} (-1)^{\ell-1} \big\{  (1-a)x\big\}^\ell  \\
&=\frac{1}{1-a^\kappa} \Big( 1- \sum_{\ell=0}^\infty \binom{\kappa}{\ell} \big\{  -(1-a)x\big\}^\ell \Big)  =\frac{1- \big( 1-(1-a)x \big)^\kappa}{1-a^\kappa}, \ \ \ 0<x<1. 
\end{split}
\end{align}
We refer
the reader to \cite{hoshino:2005}, Equation~(1.11),   and
\cite{engen:1978}, Section
3.4, for more information on the extended truncated
negative binomial distribution.

The construction proceeds conditionally on $Z$ as above. 
Given $Z=z>1$, we recursively construct discrete random vectors of
increasing dimensions, from dimension $1$ through $k$, 
as follows. 
We start with $N_0^{(1)}$, a
one-dimensional discrete random variable with,
$$
N_0^{(1)}\sim p^{\rm TNB}_{b_1/b_0,z^{-1}}.
$$
Let $2\leq m\leq k$, and suppose that we have already constructed an
$(m-1)$-dimensional discrete random vector $\bigl( N_0^{(m-1)}, \ldots, N_{m-2}^{(m-1)}\bigr)$. 
We now construct an $m$-dimensional discrete random vector $\bigl(
N_0^{(m)}, \ldots, N_{m-1}^{(m)}\bigr)$ by setting
\begin{align}
 &N_{m-1}^{(m)} \sim p^{\rm TNB}_{b_m/b_{m-1},z^{-b_{m-1}/b_0}}\label{e:ext.trun.neg.binom.setup} \\
 \notag &N_i^{(m)} = \sum_{j=1}^{N_{m-1}^{(m)}} N_{i,j}^{(m-1)}, \ i=0,\ldots,
m-2,
\end{align}
where $\bigl( N_{0,j}^{(m-1)}, \ldots, N_{m-2,j}^{(m-1)}\bigr)$, 
$j=1,2,\ldots$,    are independent of $ N_{m-1}^{(m)} $ and represent
i.i.d.~copies of the previously constructed  $\bigl( N_{0}^{(m-1)},
\ldots, N_{m-2}^{(m-1)}\bigr)$.

Once this construction is completed, the  procedure continues as
follows. We use a family of independent negative binomial random variables 
$$
T_m \sim  p^{\rm NB}_{(1+\delta)/b_m,z^{-b_m/b_0}}, 
\ \ \  m=0,1,\ldots, k. 
$$
We also use, for each $m=1,\ldots, k$, the i.i.d. copies $\bigl(
N_{0,j}^{(m)}, \ldots, N_{m-1,j}^{(m)}\bigr), \, j=1,2,\ldots $ of the
random vector $\bigl( N_{0}^{(m)}, \ldots, N_{m-1}^{(m)}\bigr)$. All
these random elements are independent (given $Z=z$).
  
We  are now ready to construct the random vector $(D_\infty^{(0)}, \dots,
 D_\infty^{(k)})$ in  \eqref{e:def.D_inf}.  

 \begin{proposition}  \label{p:representation}
    Let $Z$ be the Pareto$(\tau/b_0)$ random variable as in
  \eqref{e:def.Z}, and, conditionally on $Z=z$, define $T_m, \,
  m=0,1,\ldots, k$ and $\bigl( N_{0,j}^{(m)}, \ldots, N_{m-1,j}^{(m)}\bigr), \,
j\ge 1$, $m=1,\ldots, k$ as above. Then the random vector 
\begin{align}
  \label{e:representation.D.inf}
 \left( k-i+1 +T_i + \sum_{m=i+1}^k \sum_{j=1}^{T_m} N_{i,j}^{(m)}, \,
  i=0,\ldots, k\right)
\end{align}
has the same law as the random vector $(D_\infty^{(0)}, \dots,
 D_\infty^{(k)})$ in  \eqref{e:def.D_inf}.  Furthermore, for every
 fixed $Z=z$, and every 
 $0\leq i<m\leq k$, we have  $\sum_{j=1}^{T_m} N_{i,j}^{(m)}\sim p^{\rm
   NB}_{(1+\delta)/b_i,z^{-b_i/b_0}}$. 
\end{proposition}
\begin{remark}
It follows from Proposition \ref{p:representation} that, marginally,
for each $i=0,\ldots, k$, the limiting degree 
$D_\infty^{(i)}$ has the distribution of a negative binomial random
  variable with shape $(k-i+1)(1+\delta)/b_i$ and a random probability for
  success $Z^{-b_i/b_0}$, with $Z$ defined in \eqref{e:def.Z}. 
The negative binomial ingredients with random probabilities for
success in 
Proposition  \ref{p:representation} represent mixtures of such laws,
and such mixtures have been previously observed   in a simpler context of
preferential attachment random graphs in \cite{ ross:2013,
  samorodnitsky:resnick:towsley:davis:willis:wan:2016}.  In the case
of random simplicial complexes,  the representation is, naturally,  much
more intricate.  
\end{remark}

Our next task is to  address the multivariate power-law behaviour of the limiting
random vector $(D_\infty^{(0)}, \dots,  D_\infty^{(k)})$ in the sense
of nonstandard multivariate regular variation. The latter is defined
in terms of vague convergence of certain scaled tail probabilities.
Let $E_k = [0,\infty)^{k+1}\setminus \{ \bf 0 \}$ (so that compact
sets are closed sets bounded away from the origin) and let $M_+(E_k)$ be
the space of Radon measures on $E_k$. Given $\xi_n, \xi\in
M_+(E_k)$, we say that $\xi_n$ 
converges vaguely to $\xi$ in $M_+(E_k)$, denoted as $\xi_n
\stackrel{v}{\to} \xi$, 
 if   $\int_{E_k}g(x)\xi_n(\dif x)\to
\int_{E_k}g(x)\xi(\dif x)$ as $n\to\infty$ for all $g\in C_K^+(E_k)$,  the
space of all non-negative and continuous functions on $E_k$ with
compact support. We refer the reader to \cite{resnick:2007} for  
details on vague convergence and nonstandard multivariate regular
variation. 

Multivariate regular variation involves a tail exponent for each
coordinate, and the extremal dependence structure is characterized by
the tail measure. In order to describe the tail measure of the random
vector $(D_\infty^{(0)}, \dots,  D_\infty^{(k)})$,  we introduce another
group of random elements. Let  $\big( \Gamma_{(1+\delta)/b_m}^{(m)} \big)_{m=0}^k$ be independent Gamma random variables, where each 
$\Gamma_{(1+\delta)/b_m}^{(m)}$ has the shape parameter 
$(1+\delta)/b_m$ and the unit scale. In other words, the density function
of $\Gamma_{(1+\delta)/b_m}^{(m)}$ is given by 
$$
f_{\Gamma_{(1+\delta)/b_m}^{(m)}}(x) = \frac{1}{\Gamma\big((1+\delta)/b_m\big)}\, x^{\frac{1+\delta}{b_m}-1}e^{-x}, \ \ \ x\ge 0. 
    $$

The building blocks of the next group of random elements are certain
L\'evy processes (more specifically, subordinators), which we define
by their Laplace transforms. More specifically, 
let $\big( W_\infty^{(m,i)} (t), \, t\ge0\big)$, $1\le i \le m \le k$, be independent positive L\'evy processes, such that  
$$
\E \big[ e^{-\theta W_\infty^{(m, i)}(t)} \big] = \exp\Big\{
-t\int_0^\infty (1-e^{-\theta x}) g_i(x)\dif x\Big\}, \ \ \ \ta \ge0,
\ \  t \ge0,  
$$
where 
\begin{equation}  \label{e:levy.density}
g_i(x) := \frac{b_i}{b_{i-1}\Gamma(1-b_i/b_{i-1} )}\, x^{-1-b_i/b_{i-1}} e^{-x}, \ \ x>0, 
\end{equation}
is the corresponding L\'evy density. 
Assume further that all of the L\'evy processes above are independent of $\big( \Gamma_{(1+\delta)/b_m}^{(m)} \big)_{m=0}^k$. 
For future reference,  we note that 
\begin{align} \label{e:LogLap}
-\log \E \big[ e^{-\theta W_\infty^{(m, i)}(1)} \big] &= \int_0^\infty
  (1-e^{-\theta x})\, g_i(x)\dif x = \int_0^\theta  \int_0^\infty e^{-\xi x}x
                               g_i(x) \dif x \dif \xi  \\ 
\notag &=\frac{b_i}{b_{i-1}}\int_0^\theta  (1+\xi)^{\frac{b_i}{b_{i-1}}-1}
         \dif \xi = (1+\theta)^{b_i/b_{i-1}}-1.  
\end{align}

For $m=1,\ldots, k$,  we define  a random vector
$\BX^{(m)}$ in $[0,\infty)^{k+1}$ as follows: 
\begin{equation} \label{e:lim.proc}
X^{(m)}_m=\Gamma_{(1+\delta)/b_m}^{(m)}, \ \ \ 
  X^{(m)}_i=W_\infty^{(m, i+1)}\bigl( X^{(m)}_{i+1}\bigr), \ i=0,\ldots, m-1,
\end{equation}
and $X_i^{(m)}\equiv 0$ for  $i=m+1,\ldots, k$.

\begin{remark}
Gamma random variables appear naturally in our representation of the
tail measure as
weak limits of suitably normalized negative binomial random
variables. Similarly, 
the L\'evy processes  arise naturally as weak limits of sums of
appropriately normalized extended truncated negative binomial random
variables.  We draw the readers' attention to  similarities 
between the L\'evy density \eqref{e:levy.density}  and the 
L\'evy density of a Gamma L\'evy process. 
\end{remark}

\begin{theorem}  \label{t:regular.variation}
We have,  as $h\to\infty$, 
\begin{align*}
&h \P \bigg( \Big( \frac{D_\infty^{(m)}}{h^{b_m/\tau}} \Big)_{m=0}^k
                 \in \cdot\bigg) \stackrel{v}{\to} m_{\rm tail}
\end{align*}
in $M_+(E_k)$, with the tail measure $m_{\rm tail}$ given by              
$$
m_{\rm tail}(\cdot):= \left( \nu_{\tau/b_0} \otimes \P\circ \left(\sum_{m=0}^k
    \BX^{(m)}\right)^{-1}\right) \circ \chi^{-1}(\cdot),
$$
where $\nu_{\tau/b_0}(x,\infty) = x^{-\tau/b_0}$, $x>0$ and
$\chi:(0,\infty]\times E_k\to  E_k$ is defined as
\begin{equation}  \label{e:def.chi}
\chi(x,y_0,y_1,\dots,y_k) = (xy_0, x^{b_1/b_0}y_1,\dots, x^{b_k/b_0}y_k).
\end{equation}
In particular, for every $(x_0,x_1,\dots,x_k)\in E_k$, 
\begin{align}  \label{e:tail.explicit}
m_{\rm tail}\left( \prod_{i=0}^k (x_i,\infty] \right) 
&=\frac{\tau}{b_0}\int_0^\infty \P \Big( \sum_{m=0}^k X_i^{(m)}>\frac{x_i}{u^{b_i/b_0}}, \ i=0,\dots,k    \Big)\,  u^{-(1+\tau/b_0)}         \dif u.  
\end{align}
\end{theorem}
\begin{remark}
We will see in Remark \ref{rk:all.Gamma} that for any $i=0,\ldots, k$, 
the random variable  $\sum_{m=0}^k X_i^{(m)}$  in \eqref{e:tail.explicit}
follows the Gamma distribution with a shape parameter of
$(k-i+1)(1+\delta)/b_i$ and unit scale. It follows that the projection
of the limiting Radon measure on each axis is a mixture of Gamma distributions.
\end{remark}
\medskip

\section{Degrees of  fixed
  $k$-simplices}  \label{sec:multivariate.weak.conv}

The previous section dealt with the vector of degrees of a
$k$-simplex uniformly chosen among all $k$-simplices. In this section
we take a different point of view. We start by enumerating all
$k$-simplices ever present in the growing network. Each one of these
fixed $k$-simplices has a vector of degrees, which may change
over time. Considering all fixed $k$-simplices together gives us a
sequence of vectors of degrees, and in this section we show that this
random sequence of vectors, upon proper normalization, converges
weakly in the appropriate space. 

We begin by systematically enumerating all $k$-simplices. 
As before, we represent a $k$-simplex $\sigma$  
as an ordered set of vertices, $\sigma = \{ \sigma(0), \sigma(1),
\dots, \sigma(k) \}$ with 
$\sigma(0) < \sigma(1) < \dots < \sigma(k)$ (the initial $k+2$
vertices are numbered arbitrarily, while the subsequent vertices are
assigned numbers in order of their arrival). 
Let   
\begin{align} \label{e:initial.ksim}
\sigma_1 &= \{ 1,2,\dots,k+1 \}, \\
\notag \sigma_2 &= \{ 1,2,\dots,k,k+2 \}, \\
\notag \sigma_3 &= \{ 1,2,\dots,k-1,k+1,k+2\}, \\
\notag &\vdots \\
\notag \sigma_{k+2} &= \{ 2,3,\dots,k+2 \}, 
\end{align}
denote the $k$-simplices at time $0$ in the  model. Upon selecting the $k$-simplex at time $1$ with
label $\ell(1)$, we continue the numbering by 
\begin{align} \label{e:next.ksim}
\sigma_{k+3} &= \big\{ \sigma_{\ell(1)}\setminus \{ \sigma_{\ell(1)}(k) \}, k+3\big\}, \\
\notag \sigma_{k+4} &= \big\{ \sigma_{\ell(1)}\setminus \{ \sigma_{\ell(1)}(k-1) \}, k+3\big\}, \\
\notag &\vdots \\
\notag \sigma_{1+2(k+1)} &= \big\{ \sigma_{\ell(1)}\setminus \{ \sigma_{\ell(1)}(0) \}, k+3\big\}. 
\end{align}
More generally, by defining $\ell(n)\in \{  1,2,\dots,1+n(k+1)\}$ as the label of the selected $k$-simplex at time $n$, we extend the enumeration as follows: 
\begin{align} \label{e:gen.ksim}
\sigma_{1+n(k+1)+1} &= \big\{ \sigma_{\ell(n)}\setminus \{ \sigma_{\ell(n)}(k) \}, k+n+2\big\}, \\
\notag \sigma_{1+n(k+1)+2} &= \big\{ \sigma_{\ell(n)}\setminus \{ \sigma_{\ell(n)}(k-1) \}, k+n+2\big\}, \\
\notag &\vdots \\
\notag \sigma_{1+(n+1)(k+1)} &= \big\{ \sigma_{\ell(n)}\setminus \{ \sigma_{\ell(n)}(0) \}, k+n+2\big\}. 
\end{align}
Notice that once a $k$-simplex is assigned a number, this number will
never change.

At time $n$, all currently present $k$-simplices are numbered from 1 to
$|\mF_{k,n}|= 1 + (n+1)(k+1)$ (see \eqref{e:number.k-simplices}), 
 and, hence, determine a finite sequence of vectors of
degree counts across the dimensions:  
$$
\mD(n) := \Big( \big( D_n^{(m)}(\sigma_j(k-m),\dots,\sigma_j(k))
\big)_{m=0}^k :  1 \le j \le |\mF_{k,n}| \Big). 
$$
We view $\mD(n)$ as an element of $(\bbn^{k+1})^\infty$  by
augmenting it on the right with infinitely many zeros. We are
interested in the resulting sequence $(\mD(n))_{n\ge0}$ of random elements in 
$(\bbn^{k+1})^\infty$. We start by showing that, distributionally,
this sequence can be represented as a sequence of  functionals of
independent linear birth processes with immigration, hereafter
referred to as B.I. processes. Recall that a B.I. process with parameters $(\lambda, \theta)$, where $\lambda>0$ is a birth rate and $  \theta\geq 0$  is the  immigration rate, is a pure birth process with transition rates $q_{i,i+1}:= \lambda i +
\theta$ for $i\geq1$. For details  on  B.I. processes, we refer the
reader to \cite[Section 5]{resnick:1992} and \cite[Section
3]{wang:resnick:2019}.

We work with a  sequence of B.I. processes, in one-to-one correspondence
with the $k$-simplices ever present in the growing network, 
starting at the time a $k$-simplex joins the network. Specifically,
with the initial $k+2$ simplices in \eqref{e:initial.ksim},  we
associate   independent B.I. processes with parameters $(1,\delta)$, 
 denoted as $\big( BI_{\sigma_i}(t), \, t \geq 0 \big)_{1\leq i \leq k+2}$, such that $BI_{\sigma_i}(0)=1$.
Letting $T_1$ be the time of the first jump in one of the processes
$\big( BI_{\sigma_i}(t) \big)_{1\leq i \leq k+2}$, and we let $\ell(1)$
be the label of that process and enumerate the next $k+1$ simplices
via \eqref{e:next.ksim}. 
Once again, we associate with these simplices independent
B.I. processes with parameters $(1,\delta)$,  denoted as 
$\big( BI_{\sigma_i}(t-T_1), \, t \ge T_1 \big)_{k+3 \le i \le
  1+2(k+1)}$, such that $BI_{\sigma_i}(0)=1$. 
In general, let $T_n$ denote the time of the first jump \emph{after
  $T_{n-1}$} in one of the processes $\big( BI_{\sigma_i}(t-T_{n-1}),
\, t \ge T_{n-1} \big)_{1\le i \le 1+n(k+1)}$ and let $\ell(n)\in \{
1,2,\dots,1+n(k+1) \}$ be the label of that process. We define the
next $k+1$ simplices via \eqref{e:gen.ksim},  and their associated
B.I. processes start at that time, and the procedure continues. 

We denote  $\mathcal{H}_{k,n}=\{
\sigma_1,\sigma_2,\dots,\sigma_{1+(n+1)(k+1)} \}$, which has the same cardinality as 
$\mathcal{F}_{k,n}$, and for any $j,n$ such that 
$\sigma_j\in \mathcal{H}_{k,n}$, the label of $\sigma_j$ is denoted by
$\ell(\sigma_j)=j$ (independent of $n$ satisfying the
condition). Finally,  $b(\sigma_j)$ represents the time at which
$\sigma_j$ first joins the model. For $n=0,1,2,\dots$ 
\begin{equation}  \label{e:widetilde.D.n}
\widetilde \mD (n) := \bigg( \Big( k-m+\frac{1}{k-m+1}\hspace{-10pt}\sum_{\substack{\tau\in \mH_{k,n}, \\ \{\sigma_j(k-m), \dots, \sigma_j(k)\}\subset  \tau}} \hspace{-20pt} BI_\tau (T_n-T_{b(\tau)}) \Big)_{m=0}^k,
 \, j \in \big[ |\mH_{k,n}| \big]  \bigg) 
\end{equation}
is a finite sequence of random vectors in $\bbn^{k+1}$, which is, once
again, viewed as a random element of $(\bbn^{k+1})^\infty$ by appending 
infinitely many zeroes on the right.  

\begin{proposition}  \label{p:degree.and.BI}
 The sequences $\big( \mD(n) \big)_{n\ge0}$ and $\big( \widetilde
\mD(n) \big)_{n\ge0}$ of random elements of $(\bbn^{k+1})^\infty$ 
have the same law. 
\end{proposition}

Proposition \ref{p:degree.and.BI} is a key ingredient in establishing
the following weak convergence, which is the  main result of this
section. The limiting random element is written in terms of the jump 
times of the B.I. processes introduced above and independent Gamma
random variables. 

\begin{theorem}  \label{t:joint.weak.conv}  
Let $\tau$ be defined by \eqref{e:def.bm.tau}. Then, we  have, as $n\to\infty$, 
\begin{align*}
 &n^{-1/\tau}\mD(n)
  \Rightarrow \bigg( \Big(\frac{G^{-1/\tau}}{k-m+1}
\hspace{-25pt}\sum_{\substack{\ell\ge1, \\
    \{\sigma_j(k-m),\dots,\sigma_j(k)\}\subset  \sigma_\ell}}
\hspace{-25pt}G_\ell e^{-T_{b(\sigma_\ell)}} \Big)_{m=0}^k, \ j \ge1
\bigg) 
\end{align*}
weakly in $(\bbn^{k+1})^\infty$. Here $(G_\ell)_{\ell\ge1}$ are 
  i.i.d.$\text{Gamma} (1+\delta,1)$ random variables, 
$G$ is a $ \text{Gamma} \big(1+\delta/\tau, 1 \big)$ random
variable,  and for every $\ell\geq 1$, 
$G_\ell$ and $T_{b(\sigma_\ell)}$ are independent. 
\end{theorem}

\begin{remark}
Interestingly, the degree counts of all dimensions in Theorem
\ref{t:joint.weak.conv} grow at the same rate of $n^{1/\tau}$,
though the power tails of the limiting degree counts in Theorem
\ref{t:regular.variation} differ between different dimensions. 
\end{remark}

\begin{remark}
  It is possible to use the connection of preferential attachment
  models with B.I.~processes to obtain a description of the degree
  distribution across different dimensions. For a certain type of
  preferential attachment random graphs, this was done recently  by
  \cite{wang:resnick:2022}.  In our setting one could also use
  Proposition \ref{p:degree.and.BI} and replace the degree sequence
  in \eqref{e:def.N_n}, distributionally, with the corresponding
  B.I.~processes in \eqref{e:widetilde.D.n}.  The resulting
  description of the joint degree distribution turns out to be 
    not as explicit as that given 
  in Proposition \ref{p:representation}. Similarly, this approach does
  not seem to lead to a description of the nonstandard multivariate
  regular variation as precise as that in  Theorem
  \ref{t:regular.variation}.  
\end{remark}

\medskip

\section{Proofs}

\subsection{Proof of Theorem \ref{t:SLLN}}
As in the case of random graphs in \cite{hofstad:2017}, the claim
\eqref{e:SLLN} will follow once   Lemmas \ref{l:SLLN.Nn.exp}
and \ref{l:SLLN.exp.pmf} below are established. 

\begin{lemma}  \label{l:SLLN.Nn.exp}
For every $1\le i_k < \dots <i_0$, 
$$
\frac{N_n(i_0,\dots,i_k)}{n} - \frac{\E\big[
  N_n(i_0,\dots,i_k)\big]}{n} \to 0, \ \  \text{a.s. as } \ n\to\infty. 
$$
\end{lemma}

\begin{proof}
  The proof is nearly identical to that in the case of random graphs, where the primary tools are the Azuma-Hoeffding inequality and the Borel-Cantelli lemma; 
  see \cite{bollobas:borgs:chayes:riordan:2003,
  hofstad:2017}. 
\end{proof}

\begin{lemma}  \label{l:SLLN.exp.pmf}
For any  $1 \le i_k < \cdots <i_0$, the sequence 
$$
\big( \E \big[ N_n(i_0,\dots,i_k) \big] - n(k+1)\jointp, \ n\geq 1 \big) 
$$
is bounded. 
\end{lemma}

\begin{proof}
Fix $1\le i_k < \dots <i_0$ throughout the proof. 
We begin by deriving a recursive equation for the expected degree
counts 
$$
m_n(i_0,\dots,i_k) := \E \big[ N_n(i_0,\dots,i_k) \big]. 
$$
Specifically,  we show that for every  $n\ge1$,  
\begin{align}
\begin{split}  \label{e:recursive.mn}
m_n(i_0,\dots,i_k) &= \Big( 1-\frac{(i_0-k)b_0+\delta}{\tau n + \delta} \Big)m_{n-1}(i_0,\dots, i_k) + \frac{i_k-1+\delta}{\tau n +\delta}\, m_{n-1}(i_0-1,\dots, i_k-1)\\
&\quad +\sum_{m=0}^{k-1} \frac{(i_m-k+m-1)b_m - (i_{m+1}-k+m+1)b_{m+1}}{\tau n +\delta}\, \\
&\qquad \qquad \qquad \qquad \qquad \times m_{n-1}(i_0-1,\dots, i_m-1, i_{m+1}, \dots,i_k) \\
&\quad  + (k+1)\one \{ i_\ell=k-\ell+1, \, \ell=0,\dots,k \}. 
\end{split}
\end{align}

We begin with some simple accounting. Suppose that a new
vertex  arrives and forms a $(k+1)$-simplex with one of the existing
$k$-simplices in the model. Fix $m\in \{ 0,\dots,k \}$ and consider
an arbitrary  $m$-simplex  containing  this newly arriving
vertex. Denote this $m$-simplex as $\{v_0, \dots, v_m\}$ with $v_0 <
\dots < v_m$, where  $v_m$ is the newly arriving vertex.  Then, both of the number of  
$k$-simplices containing $\{v_0, \dots, v_m\}$,  and the sum of the $k$-degrees of
these $k$-simplices are equal to $k -  m + 1$. At a  subsequent stage, whenever
  one of the $k$-simplices containing
$\{v_0, \dots, v_m\}$ is selected by a newly arriving vertex, 
the $m$-degree of $\{v_0, \dots, v_m\}$ increases  by $1$, and  the number
of $k$-simplices containing $\{v_0, \dots, v_m\}$ increases by $k-m$,
while the sum of the $k$-degrees of these $k$-simplices  increases by  $k - m + 1$.
Therefore, if at any stage,  an $m$-simplex $\{v_0, \dots, v_m\}$ has
$m$-degree $i:=k-m+j$ for some $j\ge1$, then 
\begin{equation}  \label{e:n.k-simplices}
\text{the number of } k\text{-simplices containing } \{ v_0,\dots,v_m
\} \ \text{is} \  \ (i-k+m)(k-m)+1, 
\end{equation}
and 
\begin{equation}  \label{e:sum.k-degrees}
\text{the sum of } k\text{-degrees of all } k\text{-simplices
  containing } \{ v_0,\dots,v_m \} \ \text{is} \ \  (i-k+m)(k-m+1). 
\end{equation}

\medskip

We take $1\le i_k<i_{k-1}< \cdots <i_0$ and consider how the number of
$k$-simplices with degrees $(i_0,\dots, i_k)$ will change with the 
arrival of a new vertex at stage $n$. Several scenarios are possible.

\medskip

\noindent $(i)$ If the  $k$-simplex $\sigma=\{ v_0, \dots, v_k \}$ with $v_0 < \cdots < v_k$, existing at time $n-1$ such that $D_{n-1}^{(\ell)}(v_{k-\ell}, \dots,
v_k) = i_\ell - 1$ for each $\ell = 0, \dots, k$, is chosen by the newly arriving vertex at time $n$, then $\sigma$ and its subsimplices will increase each of its degrees by $1$. Consequently, the count of $k$-simplices with degrees $(i_0,\dots,i_k)$ increases by $1$. 
By \eqref{e:number.k-simplices} and \eqref{e:sum.k-degrees.1},
such $\sigma$  is chosen at time $n$ with probability 
$$
\frac{D_{n-1}^{(k)}(\sigma) + \delta}{\sum_{\tau \in \mathcal{F}_{k,n-1}}\big(D_{n-1}^{(k)}(\tau) + \delta \big)} = \frac{i_k - 1 + \delta}{\tau n + \delta}. 
$$
 
\medskip

\noindent $(ii)$ Let $m\in \{  0,\dots,k-1\}$ and consider a $k$-simplex $\sigma =
\{v_0, \dots, v_k\}$ with $v_0<\cdots <v_k$, existing at  time $n-1$,  such that 
$$
D_{n-1}^{(\ell)}(v_{k-\ell}, \dots, v_k)   = \begin{cases}
i_\ell-1, & \ell=0,\dots,m, \\
i_\ell, & \ell=m+1,\dots,k. 
\end{cases}
$$
If any of the
$k$-simplices existing at time $n-1$ and containing the $m$-simplex
$\{v_{k-m}, \dots, v_k\}$ but not containing $v_{k-m-1}$, is chosen by the newly
arriving vertex at time $n$, then the degrees of $\sigma$ in
dimensions $0,\ldots, m$ will increase by 1, and the degrees in
dimensions $m+1,\ldots, k$ will not change. This will, once again,
increase the count of $k$-simplices with degrees $(i_0,\ldots,
i_k)$ by 1. 

It follows from \eqref{e:n.k-simplices} and \eqref{e:sum.k-degrees}
that   the number of $k$-simplices containing the $m$-simplex
$\{v_{k-m}, \dots, v_k\}$ but not containing $v_{k-m-1}$ is  equal to 
$$
(i_m-k+m-1)(k-m) - (i_{m+1}-k+m+1)(k-m-1), 
$$
and the sum of $k$-degrees of these $k$-simplices  is 
$$
(i_m-k+m-1)(k-m+1) - (i_{m+1}-k+m+1)(k-m). 
$$
Therefore, one of these $k$-simplices is chosen at time $n$ with
probability 
\begin{align*}
&\frac{\sum_{\substack{\tau \in \mathcal{F}_{k,n-1}, \{v_{k-m},
\dots, v_k\} \subset \tau,  v_{k-m-1} \notin
\tau}}\big(D_{n-1}^{(k)}(\tau) + \delta\big)}{\sum_{\tau \in
\mathcal{F}_{k,n-1}}\big(D_{n-1}^{(k)}(\tau) + \delta\big)} =
\frac{(i_m - k + m - 1)b_m - (i_{m+1} - k + m + 1)b_{m+1}}{\tau n +
\delta}.  
\end{align*}
 
\medskip

\noindent $(iii)$ Consider a $k$-simplex $\sigma = \{v_0, \dots, v_k\}$ with $v_0<\cdots <v_k$, existing
at  time $n-1$, such that $D_{n-1}^{(\ell)}(v_{k-\ell}, \dots,
v_k) = i_\ell $ for $\ell = 0, \dots, k$. If any of the
$k$-simplices existing at time $n-1$ and containing $v_k$, 
 is chosen by the newly
arriving vertex at time $n$, then $v_k$ increases its $0$-degree by
$1$ and, hence,  the count of $k$-simplices with degrees
$(i_0,\ldots, i_k)$ decreases  by $1$. It follows from \eqref{e:n.k-simplices}
and \eqref{e:sum.k-degrees} that   the number of  $k$-simplices containing $v_k$ 
is $(i_0-k)k+1$, and the sum of their $k$-degrees is
$(i_0-k)(k+1)$. Therefore, one of these $k$-simplices is chosen at
time $n$ with 
probability 
\begin{align*}
&\frac{\sum_{\substack{\tau \in \mathcal{F}_{k,n-1},   v_{k} \in
\tau}}\big(D_{n-1}^{(k)}(\tau) + \delta\big)}{\sum_{\tau \in
\mathcal{F}_{k,n-1}}\big(D_{n-1}^{(k)}(\tau) + \delta\big)} =
\frac{(i_0- k )b_0  +\delta}{\tau n +
\delta}.  
\end{align*}

\medskip

\noindent $(iv)$ This last scenario occurs at every stage and
corresponds to the only way the count of $k$-simplices with degrees
$(k+1,k,\ldots, 1)$ increases.  Suppose at time $n$, a newly arriving vertex
$w$ forms a $(k+1)$-simplex with one of the existing $k$-simplices,
thus  introducing $\binom{k+1}{k} = k+1$ of new $k$-simplices, in
which $w$ is the  youngest vertex. Each one of these newly added $k$-simplices has
degrees $(k+1,k,\ldots, 1)$, so the count of $k$-simplices with such
degrees increases by $k+1$.  

Summarizing the above scenarios and  defining $\G_{n-1}:=\sigma(G_0,\dots,G_{n-1})$ as the $\sigma$-algebra representing the evolution of  preferential attachment processes up to time $n-1$, 
we conclude that for $1\le i_k<
\cdots < i_0$ and $n\geq 1$, 
\begin{align}
\begin{split}  \label{e:recursive.cond.exp}
&\E \big[ N_n(i_0,\dots,i_k)-N_{n-1}(i_0,\dots,i_k) \big| \mG_{n-1} \big] \\
&=  \frac{i_k-1+\delta}{\tau n + \delta }\, N_{n-1}(i_0-1,\dots,i_k-1) \\
&+\sum_{m=0}^{k-1} \frac{(i_m-k+m-1)b_m - (i_{m+1}-k+m+1)b_{m+1}}{\tau n +\delta}\, 
 N_{n-1}(i_0-1,\dots,i_m-1,i_{m+1},\dots,i_k) \\
&-\frac{(i_0-k)b_0+\delta}{\tau n +\delta}\,  N_{n-1}(i_0,\dots,i_k)  + (k+1)\one \{ i_\ell=k-\ell+1, \, \ell=0,\dots,k \}.
\end{split} 
\end{align}
Now \eqref{e:recursive.mn} follows by taking another expectation in
\eqref{e:recursive.cond.exp}.

Recall that the recursion \eqref{e:recursive.p} can be solved, so let
$(\jointp)$ be defined by that recursion. We denote 
$$
\vep_n(i_0,\dots,i_k) := m_n(i_0,\dots,i_k) - n(k+1)\jointp, \ \ \ 1\leq
i_k<\cdots<i_0, \  n\geq 0.
$$
Multiplying   both sides of \eqref{e:recursive.p} by $n(k+1)$ and  
subtracting  the resulting equation from \eqref{e:recursive.mn}, we
obtain, after some algebra, 
\begin{align}
\begin{split}  \label{e:recursive.vep.n}
\vep_n(i_0,\dots,i_k) &= \Big( 1-\frac{(i_0-k)b_0 +\delta}{\tau n + \delta} \Big)\vep_{n-1}(i_0,\dots,i_k) + \frac{i_k-1+\delta}{\tau n +\delta}\, \vep_{n-1}(i_0-1,\dots,i_k-1) \\
&+\sum_{m=0}^{k-1} \frac{(i_m-k+m-1)b_m - (i_{m+1}-k+m+1)b_{m+1}}{\tau n +\delta}\, \\
&\qquad \qquad \qquad \qquad\qquad \qquad \times \vep_{n-1}(i_0-1, \dots,i_m-1, i_{m+1}, \dots, i_k) \\
&+\frac{\big( (i_0-k)b_0 +\delta \big)(\tau +\delta)(k+1)}{\tau (\tau n +\delta)}\, \jointp - \frac{(i_k-1+\delta)(\tau +\delta)(k+1)}{\tau (\tau n +\delta)}\, p_{i_0-1,\dots,i_k-1} \\
&-\sum_{m=0}^{k-1} \frac{\big\{(i_m-k+m-1)b_m -(i_{m+1}-k+m+1)b_{m+1} \big\}(\tau+\delta)(k+1)}{\tau (\tau n +\delta)}\, \\
&\qquad \qquad\qquad\qquad\qquad\qquad\qquad\times p_{i_0-1,\dots,i_m-1,i_{m+1},\dots,i_k}, \ \ \ 1\le i_k <\cdots <i_0. 
\end{split}
\end{align}

We will use the following simple fact: if a sequence of finite numbers $(h_n)_{n\ge0}$ satisfies,  for some $\gamma, A>0$ and $n_0\in \bbn$, 
\begin{equation} \label{e:easy.rec}
|h_n|\leq (1-\gamma/n)|h_{n-1}| +A/n, 
\end{equation}
for all $n\ge n_0$, then $(h_n)_{n\ge0}$ forms a bounded sequence.  For the proof, it is enough to  choose a constant  
$$
B\ge \max\big\{ |h_0|, |h_1|,  \dots, |h_{n_0}|, A/\gamma \big\}; 
$$ 
then, an elementary inductive step  shows that 
$|h_n|\leq B$ for all $n\geq n_0$. 

Armed with this fact, we show the boundedness of the sequence
$\bigl( \vep_n(i_0,\dots,i_k)\bigr)_{n\ge0}$ for every fixed
$1\le i_k < \cdots <i_0$ by reducing the recursion \eqref{e:recursive.vep.n} 
to the form
\eqref{e:easy.rec}. We start by noticing that the terms involving
$p_{\cdots}$'s  in the right hand side of  \eqref{e:recursive.vep.n} 
have an upper bound of the form $C/n$ by an
$(i_0,\ldots,i_k)$-dependent constant $C$. To deal with the terms involving
$\vep_{\cdots}$'s  in the right hand side of  \eqref{e:recursive.vep.n}, 
except for $\vep_{n-1}(i_0,\dots,i_k)$, we use induction over
$i_0$  as follows. Start with $i_0=k+1$. Then,
for the only feasible choice of $i_1=k,\ldots, i_{k}=1$, 
there are no 
$\vep_{\cdots}$  terms in the right hand side of  \eqref{e:recursive.vep.n},  
except for  $\vep_{n-1}(k+1,k,\ldots, 1)$, so \eqref{e:recursive.vep.n} reduces
to the form \eqref{e:easy.rec}, thus giving us boundedness of the
sequence $(\vep_{n}(k+1,k,\ldots, 1))_{n\ge0}$. 

Assuming that the
sequence $(\vep_{n}(i_0,i_1,\ldots, i_{k}))_{n\ge0}$ is bounded for some
$i_0\geq k+1$ and all feasible $(i_1,\dots,i_k)$ with $1\le i_k < \cdots < i_1 < i_0$, consider the
recursion \eqref{e:recursive.vep.n} for $i_0+1$ and any feasible
$(i_1,\dots,i_k)$ with $1 \le i_k < \cdots < i_1 < i_0+1$. 
Then, all $\vep_{\cdots}$ terms in the right hand side of 
\eqref{e:recursive.vep.n}, except for $\vep_{n-1}(i_0+1,i_1,\ldots, i_{k})$, contain $i_0$ in the argument. Hence, by the induction hypothesis, all such $\vep_{\cdots}$ terms belong to bounded sequences, and \eqref{e:recursive.vep.n} reduces,
once again, to the form \eqref{e:easy.rec}. This proves the boundedness of
the sequence $(\vep_{n}(i_0+1,i_1,\ldots, i_{k}))_{n\ge0}$ for all feasible  $(i_1,\dots,i_k)$ with 
$1\le i_k< \cdots <i_1 <i_0+1$. 
\end{proof}

\begin{proof}[Proof of Theorem \ref{t:SLLN}]
 Lemmas \ref{l:SLLN.Nn.exp} and \ref{l:SLLN.exp.pmf}, imply the
convergence in \eqref{e:SLLN}  with the  limit $(\jointp)$ satisfying
the recursion \eqref{e:recursive.p}, and so  it only remains to prove
that the limit $(\jointp)$  is actually  a pmf. To this end, it is
enough to prove that the sequence of probability measures on $\{ 1\leq
i_k<\cdots<i_0\}$ defined by 
$$
\left(\frac{m_n(i_0,\ldots, i_k)}{1+(n+1)(k+1)}, \ 1\leq
i_k<\cdots<i_0\right), \ n=1,2,\ldots
$$
is tight. This will follow once we show that the sequence of the
(stochastically largest) marginal probability measures on $\{
k+1,k+2,\ldots\}$ given by 
\begin{equation}  \label{e:tight.marginal}
\left(\sum_{\substack{(i_1, \dots, i_k): \\1\leq i_k<\cdots <i_1<i_0}}\frac{m_n(i_0,i_1,\ldots,
i_k)}{1+(n+1)(k+1)}, \  i_0\geq k+1\right), \ n=1,2,\ldots
\end{equation}
is tight. According to  Lemma \ref{l:SLLN.marNn} in the Appendix,  however, 
\eqref{e:tight.marginal} converges weakly and, hence, is tight. 
\end{proof}
\medskip

\subsection{Proof of Proposition \ref{p:representation}}

We use  a high-dimensional extension of the argument in Lemma 3.1 and
Theorem 3.1 of
\cite{samorodnitsky:resnick:towsley:davis:willis:wan:2016}. We  use
the recursion \eqref{e:recursive.p} to derive a partial
differential equation (pde) for the joint pgf
\begin{equation} \label{e:pgf}
\varphi(x_0,x_1,\dots,x_k) := \sum_{\inci} \jointp \prod_{\ell=1}^k x_\ell^{i_\ell}, \quad x_\ell \in [0,1], \quad \ell=1,\dots,k.
\end{equation}
We will then solve this pde and  show that the resulting pgf coincides
with that of the random vector in 
\eqref{e:representation.D.inf}. The proof of Proposition
\ref{p:representation} is divided into three lemmas. Throughout
the proof, we set $\jointp\equiv 0$ unless $\inci$. 

\begin{lemma}  \label{l:PDE}
The joint pgf $\varphi$ in \eqref{e:pgf} 
satisfies the first-order linear pde 
\begin{equation}  \label{e:PDE.pgf}
\sum_{m=0}^k a_m(x_0,\dots,x_k) \frac{\partial \varphi}{\partial x_m}
= c(x_0,\dots,x_k) \varphi(x_0,\dots,x_k) +d(x_0,\dots,x_k),  
\end{equation}
for $x_\ell \in (0,1), \  \ell=1,\dots,k$, 
where 
\begin{align*}
a_m(x_0,\dots,x_k) &:= b_m (1-x_m)\prod_{\ell=0}^m x_\ell,  \ \ \ m=0,\dots,k, \\
c(x_0,\dots,x_k) &:= \gamma -\sum_{m=0}^k \gamma_m \prod_{\ell=0}^m x_\ell, \\
d(x_0,\dots,x_k) &:= \tau \prod_{\ell=0}^k x_\ell^{k-\ell+1}, 
\end{align*}
with $b_m$ and $\tau$ given in \eqref{e:def.bm.tau} and 
$$
\gamma := k^2-2+\delta (k^2-k-2), \ \ \gamma_m := 2(k-m) +\delta \big( 2(k-m)-1 \big), \ \ m=0,\dots,k. 
$$
\end{lemma}

\begin{proof}
Since $\varphi$ is differentiable in the interior of its domain,
\eqref{e:PDE.pgf} follows by 
multiplying both sides of \eqref{e:recursive.p}  by $\prod_{\ell=0}^k
x_\ell^{i_\ell}$, summing up over all $\inci$, and rearranging the terms.
\end{proof}

\begin{lemma}   \label{l:solu.PDE}
The joint pgf $\varphi$ can be represented as 
\begin{align}
\begin{split}  \label{e:pgf.as.solution}
\varphi(x_0,\dots,x_k) &= \prod_{m=0}^k x_m^{k-m+1} \int_1^\infty \frac{\tau}{b_0}\, z^{-(1+\tau/b_0)} \big( x_0 (1-z)+z \big)^{-\frac{1+\delta}{b_0}}  \\
&\qquad \qquad \qquad \times  \prod_{m=1}^k H_m\bigl(x_0,x_1,\ldots,
x_m, z\bigr) ^{-\frac{1+\delta}{b_m}}  \dif z,
\end{split}
\end{align}
where the functions $H_m:\, (0,1)^{m+1}\times (1,\infty)\to\R_+, \, m=1,\ldots,
k$ are defined by
\begin{align} \label{e:Hs}
  &H_1(x_0,x_1,z)=(1-x_1)z^{b_1/b_0} + x_1(x_0+z(1-x_0))^{b_1/b_0}, \\
  \notag &H_m(x_0,x_1,\ldots, x_m,z) =(1-x_m)z^{b_m/b_0} +x_m
           H_{m-1}(x_0,x_1,\ldots, x_{m-1},z)^{b_m/b_{m-1}}, \,
         m=2,\ldots, k.
\end{align}
\end{lemma} 

\begin{proof}
We solve the pde \eqref{e:PDE.pgf}  in the interior of its domain and
then use the continuity of $\varphi$ over its entire domain. 
We use the method of characteristic curves (see, e.g.,
\cite{john:1971}), which  are given as solutions of the system of
ordinary differential equations 
\begin{equation}  \label{e:char.curve}
\frac{\dif x_0}{a_0(x_0,\dots,x_k)} = \frac{\dif
  x_1}{a_1(x_0,\dots,x_k)} = \cdots = \frac{\dif
  x_k}{a_k(x_0,\dots,x_k)}. 
\end{equation}
The leftmost  equation in \eqref{e:char.curve} yields 
$$
x_1 = x_1(x_0) = \frac{1}{1+\widetilde C_1 (1-x_0)^{b_1/b_0}}, 
$$
where $\widetilde C_1>0$ is a constant. 
Similarly, for every $m=2,\dots,k$, the equation 
$$
\frac{\dif x_{m-1}}{a_{m-1}(x_0,\dots,x_k)} = \frac{\dif x_m}{a_m(x_0,\dots,x_k)}
$$
gives 
\begin{equation}  \label{e:solu.char.curve}
x_m = \frac{1}{1+\widetilde C_m (1-x_{m-1})^{b_m/b_{m-1}}}, \ \ \text{for
  some constant } \widetilde C_m>0. 
\end{equation}
The constants $(\widetilde C_m, \, m=1,\dots,k)$ define a characteristic curve 
$(x_m(x_0), \, m=1,\ldots, k)$. 
We claim that the same curve can be represented in the form  
\begin{equation}  \label{e:xm.hm}
x_m = x_m(x_0) = \frac{h_m(x_0)}{C_m + h_m(x_0)}, \ \ m=1,\dots,k,
\end{equation}
where the functions $h_m:\, (0,1)\to (0,\infty)$ and $C_m>0$,
$m=1,\dots,k$, are defined 
inductively by 
\begin{align}  
&h_1(y) = (1-y)^{-b_1/b_0}, \ \  C_1=\widetilde C_1 \ \text{ and } \label{e:def.hm} \\
&h_m (y) = \big(
C_{m-1} + h_{m-1}(y) \big)^{b_m/b_{m-1}}, \ \  C_m=\widetilde C_m C_{m-1}^{b_m/b_{m-1}}, \ \ 2 \le m \le k. \notag 
\end{align} 
The proof of \eqref{e:xm.hm} is by induction. Clearly, \eqref{e:xm.hm}
holds for $m=1$. Suppose that  \eqref{e:xm.hm} is true for some  $m \geq 1$. It then follows from  \eqref{e:solu.char.curve} and the induction hypothesis that 
$$
x_m  =   \bigg\{ 1+\widetilde C_m \Big( 1-\frac{h_{m-1}(x_0)}{C_{m-1}+h_{m-1}(x_0)} \Big)^{\frac{b_m}{b_{m-1}}} \bigg\}^{-1} = \frac{h_m(x_0)}{C_m+h_m (x_0)}, 
$$ 
by  elementary manipulations.  

For notational convenience, in the sequel we replace $x_0$ by $x$. Fix
a characteristic curve defined by \eqref{e:xm.hm} and 
define  $u(x) := \varphi \big( x,x_1(x), \dots, x_k(x) \big)$. By
Lemma \ref{l:PDE} and  \eqref{e:char.curve}, 
\begin{align*}
\frac{\dif u}{\dif x} &= \frac{\partial \varphi}{\partial x} + \sum_{m=1}^k \frac{\dif x_m}{\dif x} \cdot \frac{\partial \varphi}{\partial x_m} = \frac{c(x,x_1,\dots, x_k)}{a_0(x,x_1,\dots,x_k)}\, u(x) + \frac{d(x,x_1,\dots, x_k)}{a_0(x,x_1,\dots,x_k)}  \\
&=: \psi_1 (x)u(x) + \psi_2(x), 
\end{align*}
with $x_i = x_i(x)$ for  $i=1,\dots,k$. The solution of this
first-order linear ordinary differential equation is  
$$
u(x) = e^{\int \psi_1(x)\dif x} \Big( \int \psi_2(x)  e^{-\int \psi_1(x)\dif x } \dif x +C\Big), 
$$
where $C$ is a real-valued constant (see, e.g., \cite{rabenstein:1972}). Note that 
\begin{align}
\begin{split}  \label{e:int.psi1}
\int \psi_1(x)\dif x &= \int
\frac{c(x,x_1,\dots,x_k)}{a_0(x,x_1,\dots, x_k)}\, \dif x =\log \big\{
x^{\frac{\gamma}{b_0}} (1-x)^{\frac{\gamma_0-\gamma}{b_0}} \big\} -
\sum_{m=1}^k \frac{\gamma_m}{b_0}\, \int \frac{\prod_{\ell=1}^m
  x_\ell}{1-x}\, \dif x, 
\end{split}
\end{align}
keeping in mind that $x_i = x_i(x)$ for $i=1,\dots,k$, as before. 
We assert that for each $m=1,\dots,k$, 
\begin{equation}  \label{e:prod.integral}
\frac{1}{b_0} \int \frac{\prod_{\ell=1}^m x_\ell}{1-x}\, \dif x = \frac{1}{b_m} \log \big\{ C_m + h_m(x) \big\} +C', 
\end{equation}
for  $C'\in\R$. Indeed, it suffices to demonstrate that 
$$
\frac{1}{b_0} \cdot \frac{\prod_{\ell=1}^m x_\ell}{1-x} = \frac{1}{b_m} \cdot \frac{h_m'(x)}{C_m + h_m(x)}, 
$$
which, however, follows directly from \eqref{e:xm.hm} and
\eqref{e:def.hm} by induction. Substituting  \eqref{e:prod.integral} back into  \eqref{e:int.psi1}, we find that 
$$
\int \psi_1(x)\dif x  =\log \big\{ x^{\frac{\gamma}{b_0}}
(1-x)^{\frac{\gamma_0-\gamma}{b_0}} \big\} - \sum_{m=1}^k \log \big(
C_m + h_m(x) \big)^{\frac{\gamma_m}{b_m}} + C'', 
$$
 for $C''\in\R$. On the other hand, 
$$
\psi_2(x) = \frac{d(x,x_1,\dots,x_k)}{a_0(x,x_1,\dots,x_k)} = \frac{\tau}{b_0}\cdot \frac{x^k}{1-x} \prod_{m=1}^k \Big( \frac{h_m(x)}{C_m+h_m(x)} \Big)^{k-m+1}, 
$$
and thus, 
\begin{align*}
\int \psi_2(x)  e^{-\int \psi_1(x)} \dif x &= e^{-C''} \frac{\tau}{b_0}\int x^{k-\frac{\gamma}{b_0}} (1-x)^{\frac{\gamma-\gamma_0}{b_0}-1} \\
&\qquad \qquad \qquad \times \prod_{m=1}^k \Big\{ h_m(x)^{k-m+1} \big( C_m+h_m(x) \big)^{\frac{\gamma_m}{b_m}-(k-m+1)}  \Big\} \dif x. 
\end{align*}
Using \eqref{e:def.hm} recursively,  one can readily verify that 
$$
\prod_{m=1}^k \Big\{ h_m(x)^{k-m+1} \big( C_m+h_m(x) \big)^{\frac{\gamma_m}{b_m}-(k-m+1)}\Big\} = (1-x)^{-\frac{b_1}{b_0}k}\prod_{m=1}^k  \big( C_m+h_m(x) \big)^{-\frac{1+\delta}{b_m}}. 
$$
From this, it follows that 
\begin{align*}
u(x) &= e^{C''} x^{\frac{\gamma}{b_0}} (1-x)^{\frac{\gamma_0-\gamma}{b_0}}\prod_{m=1}^k  \big( C_m+h_m(x) \big)^{-\frac{\gamma_m}{b_m}} \\
&\quad \times \bigg\{  e^{-C''} \frac{\tau}{b_0} \int x^{k-\frac{\gamma}{b_0}} (1-x)^{\frac{\gamma-\gamma_0-b_1k}{b_0}-1} \prod_{m=1}^k \big( C_m+h_m(x) \big)^{-\frac{1+\delta}{b_m}} \dif x + C \bigg\} \\
&=C x^{\frac{\gamma}{b_0}} (1-x)^{\frac{\gamma_0-\gamma}{b_0}}\prod_{m=1}^k  \big( C_m+h_m(x) \big)^{-\frac{\gamma_m}{b_m}} + \frac{\tau}{b_0} x^{\frac{\gamma}{b_0}} (1-x)^{\frac{\gamma_0-\gamma}{b_0}}\prod_{m=1}^k  \big( C_m+h_m(x) \big)^{-\frac{\gamma_m}{b_m}} \\ 
&\qquad \qquad \qquad \qquad \qquad \qquad \qquad \times \int_0^x t^{k-\frac{\gamma}{b_0}} (1-t)^{\frac{\gamma-\gamma_0-b_1k}{b_0}-1} \prod_{m=1}^k \big( C_m+h_m(t) \big)^{-\frac{1+\delta}{b_m}} \dif t,
\end{align*}
where the final $C\in\bbr$ is a constant, not necessarily the same as
before. We now show  that this constant is equal to zero. To this end, 
write the final expression above as $u(x) = Cu_1(x)+u_2(x)$. Note that 
\begin{equation} \label{e:u.lower.bdd}
u_1(x) \  \sim c_1x^{\frac{\gamma}{b_0}}, \ u_2(x) \sim c_2 x^{k+1}, \
x\downarrow 0
\end{equation} 
for some $c_1, c_2>0$, since 
by elementary induction, each $h_m(x)$ converges to a
positive constant as $x\downarrow 0$. As $\gamma/b_0<k+1$, we conclude
that $C\geq 0$. Next, letting  $(\marp_i)_{i\ge k+1}$ be the pmf in Lemma \ref{l:dist.marp} of
the Appendix, it follows from  Lemmas  \ref{l:SLLN.exp.pmf} and 
\ref{l:SLLN.marNn} that this is the marginal pmf of the
$(k+1)$-dimensional joint pmf $(\jointp)$, corresponding to the first
coordinate. Therefore, the marginal pgf  
$$
\widetilde \varphi^{(0)}(x) := \sum_{i=k+1}^\infty \marp_i x^i, \ \ 0\le x \le 1,
$$
satisfies $u(x) \le \widetilde{\varphi}^{(0)}(x)$, $0<x<1$. By   Lemma
\ref{l:dist.marp}, we can explicitly calculate the marginal pgf as follows:  
\begin{align*}
\widetilde \varphi^{(0)}(x) &= \E\big[ x^{k+1+T_{a_0}(Z^{-1})} \big] =
                              x^{k+1} \E\Big[\big(x(1-Z) + Z
                              \big)^{-a_0}  \Big]  
\sim  \E [Z^{-a_0}]  x^{k+1} \ \ \text{ as } x\downarrow0. 
\end{align*}
 Using once again the fact that $\gamma/b_0<k+1$, this contradicts
 \eqref{e:u.lower.bdd}  unless $C=0$.

We are now in a position to derive the expression
\eqref{e:pgf.as.solution}, and to do so it remains to find what
characteristic curve a general point $(x,x_1,\dots,x_k)$ in
$(0,1)^{k+1}$ belongs to. That is, we need to find the corresponding
constants   $C_i, \, i=1,\dots,k$. An inductive argument using 
\eqref{e:xm.hm} and \eqref{e:def.hm}  shows that 
$$
\big( C_m + h_m(x) \big)^{-\frac{\gamma_m}{b_m}} =
(1-x)^{\frac{\gamma_m}{b_0}}\prod_{\ell=1}^m
x_\ell^{\frac{\gamma_m}{b_\ell}}  
$$
for every $m=1,\dots,k$. Thus, 
$$
\prod_{m=1}^k \big( C_m + h_m(x) \big)^{-\frac{\gamma_m}{b_m}} =
(1-x)^{b_0^{-1}\sum_{m=1}^k \gamma_m} \prod_{m=1}^k x_m^{b_m^{-1}
  \sum_{i=m}^k \gamma_i},  
$$
and a simple algebra shows that 
\begin{equation}  \label{e:factor}
\frac{\tau}{b_0}\,x^{\frac{\gamma}{b_0}} (1-x)^{\frac{\gamma_0-\gamma}{b_0}}\prod_{m=1}^k  \big( C_m+h_m(x) \big)^{-\frac{\gamma_m}{b_m}} = \frac{\tau}{b_0}\,x^{\frac{\gamma}{b_0}} (1-x)^{\frac{\tau}{b_0}} \prod_{m=1}^k x_m^{b_m^{-1} \sum_{i=m}^k \gamma_i}. 
\end{equation}

Next, we define functions $T_m:\, (0,1)^{m+2}\to\R_+, \, m=1,\ldots,
k$ by 
\begin{align*}
&T_1(x,x_1,t)=(1-x_1)(1-x)^{-b_1/b_0}+x_1(1-t) ^{-b_1/b_0}, \\
\notag &T_m(x,x_1,\ldots, x_m,t) = (1-x_m)(1-x)^{-b_m/b_0} + x_m
         T_{m-1}(x,x_1,\ldots, x_{m-1},t)^{b_m/b_{m-1}}, \,
         m=2,\ldots, k.
\end{align*}
Using repeatedly \eqref{e:xm.hm} and \eqref{e:def.hm}  along with an
inductive argument shows that for every $m=1,\dots,k$, 
\begin{align*}
\big( C_m+h_m(t) \big)^{-\frac{1+\delta}{b_m}}  
&=\Big(\prod_{\ell=1}^m x_\ell^{\frac{1+\delta}{b_\ell}} \Big) T_m(x,x_1,\ldots, x_m,t)^{-\frac{1+\delta}{b_m}}, 
\end{align*}
and therefore, 
\begin{equation}  \label{e:integrand}
\prod_{m=1}^k \big( C_m+h_m(t) \big)^{-\frac{1+\delta}{b_m}}=
\prod_{m=1}^k \Big[ x_m^{(1+\delta)(k-m+1)/b_m} T_m(x,x_1,\ldots, x_m,t)^{-\frac{1+\delta}{b_m}} \Big]. 
\end{equation}

Since
$$
\frac{1}{b_m} \Big( (1+\delta)(k-m+1) +\sum_{i=m}^k \gamma_i \Big) = k-m+1, \  \ \ m=1,\dots,k,
$$
it follows from \eqref{e:factor} and \eqref{e:integrand} that 
\begin{align*}
\varphi(x,x_1,\dots,x_k) &= \frac{\tau}{b_0}\, x^{\frac{\gamma}{b_0}} (1-x)^{\frac{\tau}{b_0}} \prod_{m=1}^k x_m^{k-m+1} \int_0^x t^{k-\frac{\gamma}{b_0}} (1-t)^{\frac{\gamma-\gamma_0-b_1k}{b_0}-1} \\
& \hskip 2in \times \prod_{m=1}^k T_m(x,x_1,\ldots,
                                                                                                                                                                                                           x_m,t)^{-\frac{1+\delta}{b_m}}  \dif t. 
\end{align*}
Finally, by changing the variables $t =x\big(x(1-z)+z\big)^{-1}$, 
\begin{align*}
\varphi(x,x_1,\dots,x_k) &= \frac{\tau}{b_0}\, x^{\frac{\gamma}{b_0}} (1-x)^{\frac{\tau}{b_0}} \prod_{m=1}^k x_m^{k-m+1}  \cdot x^{k-\frac{\gamma}{b_0}+1} (1-x)^{b_0^{-1}(\gamma-\gamma_0-b_1k )} \\
&\quad \times \int_1^\infty z^{-\big(1+ b_0^{-1} (\gamma_0+b_1k
                                                                                                                                                                                                                    -\gamma  )\big)} \big(x(1-z)+z    \big)^{\frac{\gamma_0+b_1k}{b_0}-(k+1)}  \\
&\qquad \qquad \times  \prod_{m=1}^k T_m\bigl(x,x_1,\ldots, x_m, x\big(x(1-z)+z\big)^{-1}\bigr) ^{-\frac{1+\delta}{b_m}} 
 \dif z  \\
&= x^{k+1}(1-x)^{-k(1+\delta)/b_0}\prod_{m=1}^k x_m^{k-m+1}\int_1^\infty \frac{\tau}{b_0}\, z^{-(1+(\tau+k(1+\delta))/b_0)} \big( x (1-z)+z \big)^{-\frac{1+\delta}{b_0}}  \\
&\qquad \qquad \qquad \times  \prod_{m=1}^k T_m\bigl(x,x_1,\ldots, x_m, x\big(x(1-z)+z\big)^{-1}\bigr) ^{-\frac{1+\delta}{b_m}}  \dif z. 
\end{align*}

Since it is easy to check by induction that
$$
T_m\bigl(x,x_1,\ldots,                  x_m,
x\big(x(1-z)+z\big)^{-1}\bigr)
= (1-x)^{-b_m/b_0} z ^{-b_m/b_0} H_m(x,x_1,\ldots,
x_m,z), \ m=1,\ldots, k,
$$
with the functions $(H_m)$ defined by \eqref{e:Hs}, the expression in
  \eqref{e:pgf.as.solution} follows by simple algebra. 
\end{proof}

\begin{lemma}  \label{l:coincidence.with.pgf}
The pgf of the random vector  \eqref{e:representation.D.inf}
coincides with  that in Lemma  \ref{l:solu.PDE}. Moreover, for every
fixed $Z=z$ and every 
 $0\leq i<m\leq k$, 
\begin{equation} \label{e:still.NB}
\E\Big[x^{\sum_{j=1}^{T_m} N_{i,j}^{(m)}}\Big] = \left[
x(1-z^{b_i/b_0})+z^{b_i/b_0}\right]^{-(1+\delta)/b_i}, \ x\in[0,1].
\end{equation}
\end{lemma}

\begin{proof}
An inspection of \eqref{e:representation.D.inf} and
\eqref{e:pgf.as.solution} shows that for the first part of the lemma, we only have to prove that for
every $z>1$, $m=1,\ldots, k$ and all $(x_0,\ldots, x_m)\in (0,1)^{m+1}$, we have
\begin{equation} \label{e:check.H}
\E \left[ x_m^{T_m}\prod_{i=0}^{m-1} x_i^{\sum_{j=1}^{T_m}
N_{i,j}^{(m)}}\right] = H_m\bigl(x_0,x_1,\ldots,
x_m, z\bigr) ^{-(1+\delta)/b_m}.
\end{equation}
It follows from \eqref{e:pgf.neg.binomial} that the
left-hand side of \eqref{e:check.H} is equal to
$$
\left[ z^{b_m/b_0} + (1-z^{b_m/b_0})x_m \E \left(\prod_{i=0}^{m-1}
x_i^{N_i^{(m)}}\right)\right]^{-(1+\delta)/b_m}. 
$$
So, by \eqref{e:Hs} it is sufficient to show that for every $m=1,\dots,k$ and $z>1$, 
$$
z^{b_m/b_0} + (1-z^{b_m/b_0})x_m  \E \Big[ \prod_{i=0}^{m-1} x_i^{N_i^{(m)}} \Big] = (1-x_m) z^{b_m/b_0} + x_m H_{m-1} (x_0, x_1,\dots,x_{m-1}, z)^{b_m/b_{m-1}}; 
$$
equivalently, 
\begin{equation}  \label{e:goal.induction}
 \E \Big[ \prod_{i=0}^{m-1} x_i^{N_i^{(m)}} \Big]  =\frac{1-z^{-b_m/b_0}H_{m-1}(x_0,x_1,\dots,x_{m-1},z)^{b_m/b_{m-1}}}{1-z^{-b_m/b_0}}. 
\end{equation}
In the above we have defined 
\begin{equation}  \label{e:def.H0}
H_0(x_0,z) := x_0 + z (1-x_0). 
\end{equation}

The proof of \eqref{e:goal.induction} is by induction on $m$. For $m=1$, we have by \eqref{e:pgf.ext.trun.neg.binomial}, 
$$
\E \big[  x_0^{N_0^{(1)}}\big] = \frac{1-\big( 1-(1-z^{-1})x_0 \big)^{b_1/b_0}}{1-z^{-b_1/b_0}} = \frac{1-z^{-b_1/b_0}H_0(x_0,z)^{b_1/b_0}}{1-z^{-b_1/b_0}}, 
$$
and so, \eqref{e:goal.induction} holds for $m=1$. Suppose now that \eqref{e:goal.induction} holds for some $1\le m \le k-1$. We have by \eqref{e:pgf.ext.trun.neg.binomial} 
and \eqref{e:ext.trun.neg.binom.setup}, 
\begin{align*}
&\E \Big[ \prod_{i=0}^{m} x_i^{N_i^{(m)}} \Big] = \E \bigg[ \Big( x_m \E \big[\prod_{i=0}^{m-1} x_i^{N_i^{(m)}}  \big] \Big)^{N_m^{(m+1)}} \bigg] \\
&= \frac{1-\big( 1-(1-z^{-b_m/b_0})x_m \E \big[  \prod_{i=0}^{m-1} x_i^{N_i^{(m)}} \big] \big)^{b_{m+1}/b_m}}{1-z^{-b_{m+1}/b_0}} \\
 &= \frac{1-z^{-b_{m+1}/b_0}H_m(x_0,\dots,x_m, z)^{b_{m+1}/b_m}}{1-z^{-b_{m+1}/b_0}}, 
\end{align*}
where in the last step we have used \eqref{e:Hs}. This  completes the inductive step and,
hence, the proof of the first part of the lemma. 

It remains to prove \eqref{e:still.NB}. It follows from \eqref{e:check.H} with $x_j=1$ for $j\neq i$ that 
$$
\E \Big[ x_i^{\sum_{j=1}^{T_m}N_{i,j}^{(m)}} \Big] = H_m(1,\dots,1,x_i,1,\dots,1,z)^{-(1+\delta)/b_m}. 
$$
Applying the recursion in \eqref{e:Hs} repeatedly, it holds that 
\begin{align*}
H_m(1,\dots,1,x_i,1,\dots,1,z)^{-(1+\delta)/b_m} = H_i(1,\dots,1, x_i,z)^{-(1+\delta)/b_i} = \big( x_i (1-z^{b_i/b_0}) + z^{b_i/b_0} \big)^{-(1+\delta)/b_i}. 
\end{align*}
This implies that $\sum_{j=1}^{T_m} N_{i,j}^{(m)}$ has a negative binomial distribution with desired parameters. 
\end{proof}

\medskip

\subsection{Proof of Theorem \ref{t:regular.variation}}

We use a technique similar to that in \cite[Theorem 4.1]{samorodnitsky:resnick:towsley:davis:willis:wan:2016}.
We start with a lemma that explains the appearance of the Gamma
random variables $\big( \Gamma_{(1+\delta)/b_m}^{(m)} \big)_{m=0}^k$
and the L\'evy processes $\big( W_\infty^{(m,i)}(t), \, t\ge0
\big)$, $1\le i \le m \le k$.  In this lemma we discuss the conditional
distributions, given $Z=z$, of the random objects appearing
in \eqref{e:representation.D.inf}, when $z\to\infty$.

\begin{lemma}  \label{l:weak.conv.neg.binom}
The following weak convergence holds: 
\begin{align*}
&\bigg(  z^{-b_i/b_0} \biggl( k-i+1 +T_i + \sum_{m=i+1}^k \sum_{j=1}^{T_m} N_{i,j}^{(m)}\biggr), \,
  i=0,\ldots, k  \bigg)  \Rightarrow  \sum_{m=0}^k \BX^{(m)}, 
\ z\to \infty, 
\end{align*}
where the $(k+1)$-dimensional random vectors $\BX^{(m)}, \, m=0,\ldots, k$
in the right-hand side are independent and  defined by \eqref{e:lim.proc}. 
\end{lemma}

\begin{proof}
We start by defining functions $g_m:\, [0,\infty)^{m+1} \to
[1,\infty), \, m=0,\ldots, k$, by 
\begin{equation} \label{g.m}
g_0(\theta_0)=1+\theta_0, \ \ \  g_m(\theta_0,\ldots,
\theta_m)=g_{m-1}(\theta_0,\ldots, \theta_{m-1})^{b_m/b_{m-1}}+ \theta_m,
\, m=1,\ldots, k.
\end{equation}
Lemma \ref{l:coincidence.with.pgf} implies that for every $z>1$ and $x_1,\dots,x_k \in [0,1]$,
\begin{equation}  \label{e:result.previous.lemma}
\E \bigg[ \bigg( \prod_{i=0}^{k-1} x_i^{T_i + \sum_{m=i+1}^k \sum_{j=1}^{T_m} N_{i,j}^{(m)}} \bigg) x_k^{T_k} \bigg] = \prod_{m=0}^k H_m(x_0,\dots,x_m, z)^{-(1+\delta)/b_m}, 
\end{equation}
where $H_m$, $m=1,\dots,k$ are given in \eqref{e:Hs} and $H_0$ is defined at \eqref{e:def.H0}.
Given $\theta_0, \dots, \theta_k \ge0$, by substituting $x_i = e^{-\theta_i z^{-b_i/b_0}}$, $i=0,\dots,k$, into \eqref{e:result.previous.lemma}, we obtain that 
\begin{align*}
&\E \bigg[\exp \bigg\{ -\sum_{i=0}^{k-1} \theta_i z^{-b_i/b_0} \Big( T_i + \sum_{m=i+1}^k \sum_{j=1}^{T_m} N_{i,j}^{(m)} \Big)-\theta_k z^{-b_k/b_0}T_k  \bigg\}  \bigg] \\
&= \prod_{m=0}^k H_m \big( e^{-\theta_0 z^{-1}}, e^{-\theta_1 z^{-b_1/b_0}}, \dots, e^{-\theta_mz^{-b_m/b_0}}, z \big)^{-(1+\delta)/b_m}. 
\end{align*}
We now claim that for every $m=0,\dots,k$, 
\begin{equation}  \label{e:induction.limit}
H_m \big( e^{-\theta_0 z^{-1}}, e^{-\theta_1 z^{-b_1/b_0}}, \dots, e^{-\theta_mz^{-b_m/b_0}}, z \big)^{-(1+\delta)/b_m} \to g_m(\theta_0,\dots,\theta_m), \ \ \text{as } z \to\infty. 
\end{equation}
The proof is by induction on $m$. 
For $m=0$, we have, as $z\to\infty$, 
$$
H_0\big( e^{-\theta_0 z^{-1}}, z \big) = e^{-\theta_0 z^{-1}} + z\big( 1-e^{-\theta_0 z^{-1}} \big) \to g_0(\theta_0). 
$$
Suppose \eqref{e:induction.limit} is true for some  $m=0,\dots,k-1$. Then, by the induction hypothesis, 
\begin{align*}
&H_{m+1} \big( e^{-\theta_0 z^{-1}}, e^{-\theta_1 z^{-b_1/b_0}}, \dots, e^{-\theta_{m+1}z^{-b_{m+1}/b_0}}, z \big)\\
&=e^{-\theta_{m+1}z^{-b_{m+1}/b_0}} \Big( H_m \big( e^{-\theta_0 z^{-1}},  \dots, e^{-\theta_mz^{-b_m/b_0}}, z \big)^{b_{m+1}/b_m} -z^{b_{m+1}/b_0} \Big) +z^{b_{m+1}/b_0} \\
&= \big(1-\theta_{m+1}z^{-b_{m+1}/b_0} +o (z^{-b_{m+1}/b_0})  \big) \Big\{ \big( g_m(\theta_0, \dots, \theta_m) +o(1)\big)^{b_{m+1}/b_m} - z^{b_{m+1}/b_0}  \Big\} + z^{b_{m+1}/b_0} \\
&\to g_m(\theta_0,\dots,\theta_m)^{b_{m+1}/b_m} + \theta_{m+1} = g_{m+1}(\theta_0,\dots,\theta_{m+1}), \ \  \ z\to\infty. 
\end{align*}
Thus, we  conclude that 
\begin{align*}
&\lim_{z\to\infty} \E \bigg[  \exp \bigg\{ -\sum_{i=0}^{k-1} \theta_i z^{-b_i/b_0} \Big( k-i+1+T_i + \sum_{m=i+1}^k \sum_{j=1}^{T_m} N_{i,j}^{(m)} \Big)-\theta_k z^{-b_k/b_0}(1+T_k)  \bigg\}  \bigg]\\
&=\lim_{z\to\infty} \E \bigg[  \exp \bigg\{ -\sum_{i=0}^{k-1} \theta_i z^{-b_i/b_0} \Big( T_i + \sum_{m=i+1}^k \sum_{j=1}^{T_m} N_{i,j}^{(m)} \Big)-\theta_k z^{-b_k/b_0}T_k  \bigg\}  \bigg]\\
&  = \prod_{m=0}^k g_m(\theta_0,\dots,\theta_m)^{-(1+\delta)/b_m}. 
\end{align*}

In order to complete the proof of the lemma, it is 
sufficient to prove that for any $m=0,\ldots, k$, we have
\begin{equation} \label{e:correct,Lapl}
\E\Big[ \exp\big\{ -\sum_{i=0}^m \theta_i X_i^{(m)}\big\}\Big] =
g_m(\theta_0,\ldots,\theta_m)^{-(1+\delta)/b_m}.
\end{equation}

The crucial observation is that
   the $(m+1)$-dimensional stochastic process denoted by $\BY^{(m)}(t)=(Y_0^{(m)}(t),\ldots,
   Y_{m}^{(m)}(t))$, $ t\geq 0$, which we define by 
$Y_{m}^{(m)}(t)=t$, and $Y_i^{(m)}(t)=W_\infty^{(m,i+1)}(Y_{i+1}^{(m)}(t)), \,
   i=0,\ldots, m-1, \, t\ge0$ is an $(m+1)$-dimensional L\'evy
   process.  Furthermore, $\BX^{(m)}\eid \BY^{(m)}\bigl(
   \Gamma^{(m)}_{(1+\delta)/b_m}\bigr)$. 
Therefore, by the properties of a L\'evy process, 
   \begin{align*}
\E\left[ \exp\left\{ -\sum_{i=0}^m \theta_i X_i^{(m)}\right\}\right] &= \E\left[\exp\left\{
     -\sum_{i=0}^m \theta_i  Y_i^{(m)}\bigl(
   \Gamma^{(m)}_{(1+\delta)/b_m}\bigr)\right\}\right]\\
&= \E\left[ \left( \E\left[ \exp\left\{
     -\sum_{i=0}^m \theta_i  Y_i^{(m)}(1)\right\}\right]\right)^{ \Gamma^{(m)}_{(1+\delta)/b_m}}\right]\\
&= \left( 1+\theta_m - \log \E\left[ \exp\left\{
     -\sum_{i=0}^{m-1} \theta_i
   Y_i^{(m)}(1)\right\}\right]\right)^{-(1+\delta)/b_m}. 
   \end{align*}
  Therefore, the statement \eqref{e:correct,Lapl} is equivalent to 
  \begin{equation} \label{e:correct,Lapl1}
1+\theta_m - \log \E\left[ \exp\left\{
     -\sum_{i=0}^{m-1} \theta_i    Y_i^{(m)}(1)\right\}\right] =
   g_m(\theta_0,\ldots,\theta_m)
 \end{equation}
 for every $m=0,\ldots, k$.     
Once again, we prove \eqref{e:correct,Lapl1}  by induction in $m$. 
For $m=0$, \eqref{e:correct,Lapl1} is obviously true. 
Assume that \eqref{e:correct,Lapl1} holds for some $m=0,\ldots, k-1$. We will
   prove that it also holds for $m+1$. Using, once again, the
   properties of a L\'evy process and \eqref{e:LogLap}, 
   \begin{align*}
&\E\left[\exp\left\{
     -\sum_{i=0}^{m} \theta_i    Y_i^{(m+1)}(1)\right\}\right]
  = \E\left[ \E\left[\exp\left\{
     -\sum_{i=0}^{m} \theta_i    Y_i^{(m+1)}(1)\right\}\right]\Bigg|Y_m^{(m+1)}(1)\right]\\
=&\E\left[ \left\{ e^{-\theta_m} \E\left[\exp\left\{
     -\sum_{i=0}^{m-1} \theta_i    Y_i^{(m)}(1)\right\}\right]\right\}^{Y_m^{(m+1)}(1)}\right]\\
=& \E\left[\exp\left\{ -\left( \theta_m - \log \E\left[\exp\left\{
     -\sum_{i=0}^{m-1} \theta_i    Y_i^{(m)}(1)\right\}\right]\right) Y_m^{(m+1)}(1)\right\}\right]\\
=&\exp\left\{ -\left[ \left( 1+\theta_m - \log \E\left[\exp\left\{
     -\sum_{i=0}^{m-1} \theta_i    Y_i^{(m)}(1)\right\}\right]\right)^{b_{m+1}/b_m}-1\right]\right\}\\
=&\exp\left\{ -\bigl( g_m(\theta_0,\ldots, \theta_m) ^{b_{m+1}/b_m}-1\bigr)\right\}
  \end{align*}
  by the assumption of the induction. Therefore,
  \begin{align*}
1+\theta_{m+1} - \log \E\left[\exp\left\{
     -\sum_{i=0}^{m} \theta_i    Y_i^{(m+1)}(1)\right\}\right] &=\theta_{m+1}+
   g_m(\theta_0,\ldots,\theta_m) ^{b_{m+1}/b_m} \\
   &=  g_{m+1}(\theta_0,\ldots,\theta_{m+1}). 
  \end{align*}
This completes the induction step, hence proving
\eqref{e:correct,Lapl1} and finishing the proof of the lemma. 
\end{proof}
\medskip

\begin{remark} \label{rk:all.Gamma}
Note that for each $0\le i \le m \le k$, we have by \eqref{g.m} and 
\eqref{e:correct,Lapl}, 
\begin{align*}
\E[e^{-\theta_i X_i^{(m)}}] = g_m(0,\ldots, 0,\theta_i,0,\ldots,
0)^{-(1+\delta)/b_m}
= g_i(0,\ldots, 0,\theta_i)^{-(1+\delta)/b_i} = (1+\theta_i)
^{-(1+\delta)/b_i}, 
\end{align*}
so each $X_i^{(m)}$ has the Gamma distribution with a shape parameter of
$(1+\delta)/b_i$ and unit scale. Furthermore, by independence, $\sum_{m=0}^k X_i^{(m)}$ has a Gamma distribution as well with a shape parameter of $(k-i+1)(1+\delta)/b_i$ and unit scale. 
\end{remark}

We now translate the statement of Lemma \ref{l:weak.conv.neg.binom}
into a vague  convergence statement by ``mixing'' its statement over
the values of $Z$. 
\begin{lemma}  \label{l:first.step.vague.conv}
We have, as   $h\to\infty$, 
\begin{align*}
&h\P\left( \left( \frac{Z}{h^{b_0/\tau}}, \, \bigg(  Z^{-b_i/b_0} \biggl( k-i+1 +T_i + \sum_{m=i+1}^k \sum_{j=1}^{T_m} N_{i,j}^{(m)}\biggr), \,
  i=0,\ldots, k  \bigg) \right)\in \cdot\right) \\
&\qquad \stackrel{v}{\to}\left( \nu_{\tau/b_0} \otimes \P\circ \left(\sum_{m=0}^k \BX^{(m)}\right)^{-1}\right)(\cdot)
\end{align*}
in $M_+\big( (0,\infty]\times  E_k\big)$. 
\end{lemma}

\begin{proof}
It is sufficient to prove that for every $x>0$ and a continuous  bounded function $g:[0,\infty)^{k+1}\to \R$, 
\begin{align*}
&h\E \bigg[ \one \Big\{ \frac{Z}{h^{b_0/\tau}} >x \Big\} \, g \bigg(  Z^{-b_i/b_0} \biggl( k-i+1 +T_i + \sum_{m=i+1}^k \sum_{j=1}^{T_m} N_{i,j}^{(m)}\biggr), \,  i=0,\ldots, k  \bigg) 
 \bigg] \\
&\to x^{-\tau/b_0}\, \E \bigg[ g\bigg(  \sum_{m=0}^k \BX^{(m)} \bigg)
             \bigg], \  h\to\infty. 
\end{align*}
This, however, follows from   Lemma \ref{l:weak.conv.neg.binom} and
the fact that $Z$ is  a Pareto $(\tau/b_0)$ random variable (the argument
is spelled out in the proof of \cite[Theorem 4.1]{samorodnitsky:resnick:towsley:davis:willis:wan:2016}).
\end{proof}

\medskip

\begin{proof}[Proof of Theorem \ref{t:regular.variation}]
  We start by extending the function $\chi$ in \eqref{e:def.chi} to
  the function  $\chi_1:(0,\infty]\times  [0,\infty]^{k+1}\to (0,\infty]\times [0,\infty]^{k+1} $ defined as
$$
\chi_1(x,y_0,y_1,\dots,y_k) = (x,xy_0, x^{b_1/b_0}y_1,\dots,
x^{b_k/b_0}y_k).  
$$
Since the function $\chi_1$ obviously satisfies the compactness condition in
Proposition 5.5 of \cite{resnick:2007},
it follows from that proposition and  Lemma
\ref{l:first.step.vague.conv} that 
\begin{align} \label{e:vague.0}
&h \P \left( \left( \frac{Z}{h^{b_0/\tau}},\,   \left(  h^{-b_i/\tau} \biggl( k-i+1 +T_i + \sum_{m=i+1}^k \sum_{j=1}^{T_m} N_{i,j}^{(m)}\biggr), \,
  i=0,\ldots, k  \right)
\right) \in \cdot \right)  \\
\notag 
&\qquad \stackrel{v}{\to} \left( \nu_{\tau/b_0} \otimes \P\circ \left(\sum_{m=0}^k
    \BX^{(m)}\right)^{-1}\right) \circ (\chi_1)^{-1}(\cdot)
\end{align}
in $M_+\big( (0,\infty] \times [0,\infty]^{k+1} \big)$. We will show
that this implies the vague convergence  
\begin{align} \label{e:reduced.vague}
&h  \P \left(   \left(  h^{-b_i/\tau} \biggl( k-i+1 +T_i + \sum_{m=i+1}^k \sum_{j=1}^{T_m} N_{i,j}^{(m)}\biggr), \,
  i=0,\ldots, k  \right)
\in \cdot \right)   \\
\notag 
&\qquad  \stackrel{v}{\to} \left( \nu_{\tau/b_0} \otimes \P\circ \left(\sum_{m=0}^k
    \BX^{(m)}\right)^{-1}\right) \circ \chi^{-1}(\cdot)
\end{align}
in $M_+( E_k )$, which is
equivalent to the assertion of the theorem. 

As in a similar situation in  \cite[Theorem
4.1]{samorodnitsky:resnick:towsley:davis:willis:wan:2016}, a
Slutsky-type argument shows that in order to derive
\eqref{e:reduced.vague} from \eqref{e:vague.0},  it is enough to prove
that for any $i=0,\ldots, k$ and $\vep>0$, 
$$
\lim_{x\to0} \limsup_{h\to\infty} h \P \left( \frac{Z}{h^{b_0/\tau}} 
\le x, \,   h^{-b_i/\tau} \biggl( k-i+1 +T_i + \sum_{m=i+1}^k \sum_{j=1}^{T_m} N_{i,j}^{(m)}\biggr)>\vep
 \right)=0. 
$$
Recall from Proposition \ref{p:representation} that  $\sum_{j=1}^{T_m}N_{i,j}^{(m)}$ has the same negative binomial distribution as $T_i$ for each $m=i+1,\dots,k$, so it is sufficient to verify that for every $\vep>0$, 
$$
\lim_{x\to0} \limsup_{h\to\infty} h \P \left( \frac{Z}{h^{b_0/\tau}} 
\le x, \, \frac{T_i}{h^{b_i/\tau}} >\vep  \right) = 0. 
$$
By (4.13) in \cite{samorodnitsky:resnick:towsley:davis:willis:wan:2016} and Markov's inequality, 
$$
P \Big( \frac{T_i}{h^{b_i/\tau}} >\vep\, \Big| \, Z=z \Big) \le \frac{\E[T_i^\ell | Z=z]}{\vep^\ell h^{\ell b_i/\tau}} \le \frac{C}{\vep^\ell h^{\ell b_i /\tau}}\, z^{\ell b_i/b_0} 
$$
for some $C>0$, where $\ell$ is a positive integer satisfying $\ell>\tau/b_i$. Using this bound, 
\begin{align*}
h \P \left( \frac{Z}{h^{b_0/\tau}} 
\le x, \, \frac{T_i}{h^{b_i/\tau}} >\vep  \right)  
&\le C\vep^{-\ell} h^{1-\ell b_i/\tau} \int_1^{xh^{b_0/\tau}}
z^{\ell b_i/b_0} z^{-(1+\tau/b_0)}\, \dif z \\
&\le Cx^{-(\tau/b_0)+(\ell b_i/b_0)}\to 0, \ \  \text{as } x \to 0, 
\end{align*}
by the choice of $\ell$ (here, the constant $C$ varies  from line to line). This completes the proof of the theorem.  
\end{proof}

\medskip

\subsection{Proof of Proposition \ref{p:degree.and.BI}}

Before commencing the proof, we establish certain properties of the
independent B.I.~processes appearing in the proposition. 
These properties are analogous to
\eqref{e:sum.k-degrees.1}, \eqref{e:n.k-simplices}, and
\eqref{e:sum.k-degrees}, respectively. The first property is 
\begin{equation}  \label{e:sum.BI.proc}
\sum_{\tau\in \mH_{k,n}}BI_\tau (T_n-T_{b(\tau)}) = (n+1)(k+2), 
\end{equation}
which is analogous to  \eqref{e:sum.k-degrees.1}, and both hold for
the same reason. Indeed,  note that the construction in Section
\ref{sec:multivariate.weak.conv} begins with $k+2$ B.I.~processes at
time $0$, each with an initial value of $1$.  Subsequently, at time
$T_j$ for $j\in \{ 1,\dots,n \}$, one of the existing processes $\big(
BI_{\sigma_\ell}(t) \big)_{1 \le \ell \le 1+j(k+1)}$ has an upward
jump of size  $1$. Meanwhile, $k+1$ new B.I.~processes are initiated
with initial values all equal to $1$. Consequently, the left-hand side
in \eqref{e:sum.BI.proc} increases by $k+2$ at each $T_j$, thus
verifying \eqref{e:sum.BI.proc}.

We proceed with two relations analogous to \eqref{e:n.k-simplices} and
\eqref{e:sum.k-degrees}.  
Suppose a B.I.~process $\big( BI_\sigma(t) \big)$, where $\sigma=\{
\sigma(0), \dots, \sigma(k) \}$ with $\sigma(0)<\dots < \sigma(k)$,
arises in the construction due to an addition of $\sigma(k)$. At that
time 
the number of the B.I.~processes whose indices contain $\{\sigma(k-m),
\dots, \sigma(k)\}$ is $k-m+1$, and the sum of their values is also
$k-m+1$. Therefore, initially, 
\begin{align}  
k-m+ \frac{1}{k-m+1} \times &\text{  the sum of the values of the B.I.~processes } \label{e:relate.to.degree.0}\\
&\qquad \text{whose indices contain } \{\sigma(k-m),\dots,\sigma(k)\} = k-m+1. \notag
\end{align}
When  one of the B.I.~processes, whose indices contain
$\{\sigma(k-m),\dots,\sigma(k)\}$, jumps for the first time, the
number of  such B.I.~processes (i.e. those containing $\{\sigma(k-m),\dots,\sigma(k)\}$
in their indices) increases to $2(k-m)+1$.   
Furthermore, each of the $k-m$ newly added such B.I.~processes has
initial value of $1$, while the value of the B.I.~process, which has
jumped,   increases by $1$. Overall,   the sum of the values of the
B.I.~processes, whose indices contain
$\{\sigma(k-m),\dots,\sigma(k)\}$ becomes $2(k-m+1)$, and the
expression in 
the left-hand side in \eqref{e:relate.to.degree.0} becomes
$k-m+2$.  In general, after $i-(k-m+1)\geq 0$ jumps in 
the B.I.~processes, whose indices contain
$\{\sigma(k-m),\dots,\sigma(k)\}$, 
\begin{align*} 
&\text{the number of B.I.~processes whose indices contain }   \\
&\qquad \qquad \{\sigma(k-m),\dots,\sigma(k)\}  = (i-k+m)(k-m)+1, \notag
\end{align*}
and 
\begin{align*}  
&\text{the sum of values of B.I.~processes whose indices contain }   \\
&\qquad \qquad \{\sigma(k-m),\dots,\sigma(k)\}  = (i-k+m)(k-m+1).\notag
\end{align*}
Furthermore, at that time 
 the left-hand side in \eqref{e:relate.to.degree.0}
is equal to $i$. 
 
\begin{proof}[Proof of Proposition \ref{p:degree.and.BI}]
Note that 
$$
\mD(0) = \big( (k-m+1)_{m=0}^k, \dots, (k-m+1)_{m=0}^k \big) =
\widetilde \mD(0). 
$$
We enlarge the state space of both processes by including the
identities of the simplices (i.e., the list of simplices) present at each time, in addition to their
degrees, in the state of the processes. We use the notation $\mD^*$
and $\widetilde \mD^*$ for the extended processes.  Clearly, the
extended processes 
still start in the same state at time 0. Since the two processes are
both Markovian by construction, it is enough 
to show that the processes have identical
transition probabilities. To this end, we assume that the two processes are in the same state at some time
$n\geq 0$, and establish a
one-to-one correspondence between events that can occur at this time
in the two systems, and show that, under this correspondence, the
events will have the same probabilities, and the two
processes will transit into identical states. 

Recall from \eqref{e:number.k-simplices} that $|\mF_{k,n}|=|\mH_{k,n}|=1+(n+1)(k+1)$, and 
consider $(c_\ell^{(0)}, \dots, c_\ell^{(k)})_{\ell=1, \dots, |\mF_{k,n}|}$ with $c_\ell^{(m)}\in \bbn$, $m=0,\dots,k$, such  that 
$$
\P \Big( \mD (n)=(c_\ell^{(0)},\dots, c_\ell^{(k)})_{\ell=1, \dots, |\mF_{k,n}|}  \Big)>0. 
$$ 
It follows from  \eqref{e:sum.k-degrees.1} that 
\begin{equation}  \label{e:sum.k-degree}
\sum_{\ell\in [|\mF_{k,n}|]}c_\ell^{(k)}=(n+1)(k+2). 
\end{equation}

Whenever $\mD (n)=(c_\ell^{(0)},\dots, c_\ell^{(k)})_{\ell=1, \dots,
|\mF_{k,n}|}$ and the identities of the $k$-simplices
$\sigma_\ell, \ell \in \big[ |\mF_{k,n}| \big]$, are given,  one of the events
$B_j, \, j \in \big[ |\mF_{k,n}| 
\big]$, will happen, where $B_j$ denotes the event that the next arriving
vertex attaches itself to $\sigma_j$. The probability of this event
(given the state) is $(c_j^{(k)}+\delta)/\bigl(
n(k+2)+\delta(1+n(k+1))\bigr)$ (see \eqref{e:prob.select.sigma}). When this event happens, the extended
process $\mD^*$ will move to the following state. First of all, the
list of the $k$ simplices will be augmented by the $k+1$ new
$k$-simplices consisting of the new vertex and $k$ of the $k+1$
vertices of $\sigma_j$. The new $k$-simplices will have identical
degrees $(k-m+1)_{m=0}^k$. Finally, for any $m=0,\ldots, k$, and
any existing $k$-simplex $\sigma_\ell,   \ell \in \big[ |\mF_{k,n}| \big]$
such that  $\{\sigma_j(k-m), \dots,\sigma_j(k)\}\subset  \sigma_\ell$,
the entry  $c_\ell^{(m)}$ is increased by 1, because of the addition of a newly arriving vertex.  

Similarly, whenever $\widetilde\mD (n)=(c_\ell^{(0)},\dots, c_\ell^{(k)})_{\ell=1, \dots,
|\mH_{k,n}|}$ and  the identities of the $k$-simplices
$\sigma_\ell, \ell \in \big[ |\mH_{k,n}| \big]$, are given, one of the events
$\widetilde B_j, \, j \in \big[ |\mH_{k,n}| 
\big]$, will happen, where $\widetilde B_j$ is  the event that the
B.I. process $BI_{\sigma_j}$ has an upward jump at time $T_{n+1}$. Observe that the probability of this event
(given the state) is again $(c_j^{(k)}+\delta)/\bigl(
n(k+2)+\delta(1+n(k+1))\bigr)$, as a result of  the standard competition
of exponential clocks and \eqref{e:sum.k-degree}. By construction, 
the
list of the $k$ simplices will again be augmented by the $k+1$ new
$k$-simplices consisting of the new vertex and $k$ of the $k+1$
vertices of $\sigma_j$. Furthermore, the new $k$-simplices will have identical
degrees $(k-m+1)_{m=0}^k$ and, by \eqref{e:widetilde.D.n}, these same
vectors will augment the state of the process $\widetilde\mD^*$. Also
by \eqref{e:widetilde.D.n}, for any $m=0,\ldots, k$, and
any existing $k$-simplex $\sigma_\ell,   \ell \in \big[ |\mH_{k,n}| \big]$
such that  $\{\sigma_j(k-m), \dots,\sigma_j(k)\}\subset  \sigma_\ell$,
the entry  $c_\ell^{(m)}$ is increased by 1. 

This establishes the required correspondence between the transitions
of the two processes and, hence, completes the proof. 
\end{proof}
\medskip

\subsection{Proof of Theorem \ref{t:joint.weak.conv}}

\begin{proof}[Proof of Theorem \ref{t:joint.weak.conv}]
The process 
$$
BI^\ms{com}(t):= \frac{1}{k+2}\sum_{\ell=1}^\infty BI_{\sigma_\ell}(t-T_{b(\sigma_\ell)})\, \one \{ t\ge T_{b(\sigma_\ell)} \}, 
$$
is easily seen to be a B.I.~process  
with parameters $(\tau, \delta)$, such that
$BI^\ms{com}(0)=1$. It is a standard fact about linear B.I. processes that,
as $t\to\infty$, 
$$
e^{-\tau t}BI^\ms{com}(t) \to G\sim \Gamma(1+\delta/\tau,1) \ \
\text{a.s.}, 
 $$
see e.g.  \cite[Theorem 1]{wang:resnick:2019} and Remark 2
therein. Since by   \eqref{e:sum.BI.proc}, we have 
$BI^\ms{com}(T_n) = n+1 $ for each $n$, 
and  $T_n\to\infty$ almost surely as $n\to\infty$, we conclude that 
$$
e^{-\tau T_n}(n+1) = e^{-\tau T_n} 
BI^\ms{com}(T_n)  \to G\ \ \text{a.s.}
$$
as $n\to\infty$, which is  equivalent to 
\begin{equation}  \label{e:scaling.const}
n^{1/\tau} e^{-T_n} \to G^{1/\tau}  \ \ \text{a.s.}, \ n\to\infty.
\end{equation}
Appealing to  Proposition \ref{p:degree.and.BI}, we have 
\begin{align*}
&\bigg( \Big( \frac{D_n^{(m)}(\sigma_j(k-m), \dots, \sigma_j(k))}{n^{1/\tau}} \Big)_{m=0}^k, \ j \in \big[ |\mF_{k,n}| \big] \bigg)_{n\ge0}  \\
&\stackrel{d}{=}\bigg( \bigg( \frac{k-m}{n^{1/\tau}}+\frac{1}{(k-m+1)n^{1/\tau}} \hspace{-25pt}\sum_{\substack{\ell \ge 1, \\ \{\sigma_j(k-m), \dots, \sigma_j(k)\}\subset  \sigma_\ell}} \hspace{-35pt} BI_{\sigma_\ell} (T_n-T_{b(\sigma_\ell)})\, \one \{ T_n\ge T_{b(\sigma_\ell)} \}\bigg)_{m=0}^k, 
 \ j \in \big[ |\mH_{k,n}| \bigg)_{n\ge0}. 
\end{align*}

Since for each $\ell\ge1$,  $BI_{\sigma_\ell}$ is a B.I.~process with
parameters $(1,\delta)$ and $BI_{\sigma_\ell}(0)=1$, it follows, once
again from \cite[Theorem 1]{wang:resnick:2019}, that, as $t\to\infty$, 
$$
e^{-t} BI_{\sigma_\ell}(t)  \to G_\ell\sim \Gamma(1+\delta,1)  \ \ \text{a.s.}
$$
As $(BI_{\sigma_\ell})_{\ell\ge1}$ are i.i.d., so are $(G_\ell)_{\ell\ge1}$. Since $T_n\to\infty$ almost surely,  we have, for every $\ell\ge1$,
$$
\frac{BI_{\sigma_\ell}(T_n-T_{b(\sigma_\ell)})}{e^{T_n}} \to G_\ell
e^{-T_{b(\sigma_\ell)}} \ \text{a.s.}
$$
as $n\to\infty$. Automatically, 
$G_\ell$ and $T_{b(\sigma_\ell)}$ are independent for each $\ell\geq
1$. 
Now, the claim of the
theorem follows from \eqref{e:scaling.const}. 
\end{proof}
\medskip

\section{Appendix}

Consider the sequence $(\marp_i)_{i\ge k+1}$, defined recursively by 
\begin{equation}  \label{e:recursive.mar.p}
\marp_i = \frac{(i-k-1)b_0+\delta}{(i-k)b_0 +(k+2)(1+\delta)}\, \marp_{i-1} + \frac{\tau}{b_0 + (k+2)(1+\delta)}\, \one \{ i=k+1 \}, \ \ \ i\ge k+1, 
\end{equation}
where $b_0$ and $\tau$ are defined in \eqref{e:def.bm.tau}. The
following lemma verifies that this sequence is a pmf on $\{
k+1,k+2,\ldots\}$. 

\begin{lemma}  \label{l:dist.marp}
Let $Z$ be the Pareto random variable in \eqref{e:def.Z} and let $T_{a_0}(Z^{-1})$ be the negative binomial
random variable  \eqref{e:def.neg.binomial} with shape $a_0
:= (k+1)(1+\delta)/b_0$ and random probability for success $Z^{-1}$. Then for
any $i\geq k+1$, 
\begin{equation}  \label{e:D.inf.T.Z}
\P\big(k+1 +T_{a_0}(Z^{-1})=i\big) = \marp_i.
\end{equation} 
\end{lemma}

\begin{proof}
One can readily solve the recursive equation \eqref{e:recursive.mar.p}: 
\begin{equation}  \label{e:sol.mar.p}
\marp_i =\frac{\tau}{b_0 + (k+2)(1+\delta)} \cdot \frac{\Gamma(\beta_{k+2})}{\Gamma(\alpha_{k+2})}\cdot \frac{\Gamma(\alpha_{i+1})}{\Gamma (\beta_{i+1})}, \ \ \ i\ge k+1, 
\end{equation}
with $\alpha_i := i-k-1 +\delta/b_0$ and $\beta_i :=
i-k+(k+2)(1+\delta)/b_0$ for  $i\geq k+2$. Now,  an
elementary calculation shows that  for every $i\ge k+1$, 
\begin{align*}
\P \big( k+1+T_{a_0}(Z^{-1})=i\big) &=\frac{\Gamma(i-k-1+a_0)}{(i-k-1)!\, \Gamma (a_0)}\int_1^\infty (1-z^{-1})^{i-k-1} z^{-a_0}  \frac{\tau}{b_0} z^{-(1+\tau/b_0)}\dif z \\
 &= \frac{\tau}{b_0}\cdot  \frac{\Gamma(a_0+\tau/b_0)}{\Gamma (a_0)}\cdot \frac{\Gamma(i-k-1+a_0)}{\Gamma(i-k+a_0+\tau/b_0)}. 
\end{align*}
It is straightforward to verify that the final expression above is
equal to  the right-hand side of \eqref{e:sol.mar.p},
so \eqref{e:D.inf.T.Z} follows. 
\end{proof}

\begin{lemma}  \label{l:SLLN.marNn}
 For $ i\ge k+1$,  let 
 \begin{align*}
\marm_n(i) :=& \E \bigg[ \sum_{ \substack{(i_1,\dots,i_k): \\ 1\le i_k < \dots <i_1
    <i}} N_n (i, i_1,\dots,i_k) \bigg]
   = \sum_{\substack{(i_1,\dots,i_k): \\ 1\leq i_ k<\cdots <i_1<i}}m_n(i,i_1,\ldots,
i_k). 
  \end{align*}
  Then,  for the sequence of probability measures
  $\bigl(\marm_n(i)/(1+(n+1)(k+1)), \, i\geq k+1\bigr)$, 
it holds that for every $i\ge k+1$, 
$$
\frac{\marm (i)}{1+(n+1)(k+1)}\to \marp_i \ \ \text{as } n\to\infty.
$$
\end{lemma}

\begin{proof}
 The consideration of the scenarios used in the derivation of
 \eqref{e:recursive.mn} establishes a (simpler) recursion 
\begin{align}
\begin{split}  \label{e:recursive.marm}
\marm_n(i) &= \Big( 1-\frac{(i-k)b_0+\delta}{\tau n +\delta} \Big) \marm_{n-1}(i)  + \frac{(i-k-1)b_0 +\delta}{\tau n +\delta}\, \marm_{n-1}(i-1)  + (k+1)\one \{ i=k+1 \}. 
\end{split}
\end{align}
Let $\marvep_n (i) := \marm_n(i) -n(k+1)\marp_i$. 
Multiplying both sides of the recursion  \eqref{e:recursive.mar.p} by
$n(k+1)$ and subtracting it from \eqref{e:recursive.marm} gives us 
\begin{align*}
\marvep_n(i) &= \Big( 1-\frac{(i-k)b_0+\delta}{\tau n +\delta} \Big)\marvep_{n-1}(i) + \frac{(i-k-1)b_0+\delta}{\tau n +\delta}\, \marvep_{n-1}(i-1) \\
&\quad +\frac{\big( (i-k)b_0+\delta \big)(\tau +\delta)(k+1)}{\tau (\tau n +\delta)}\, \marp_i - \frac{\big( (i-k-1)b_0+\delta \big)(\tau +\delta)(k+1)}{\tau (\tau n +\delta)}\, \marp_{i-1}. 
\end{align*}
An  argument similar to (but simpler than) that  used to
establish boundedness of the sequence $\bigl( \vep_n(i_0,\dots,i_k)\bigr)_{n\ge0}$
in the proof of Lemma \ref{l:SLLN.exp.pmf}, shows that for every $i \ge k+1$, the 
sequence $(\marvep_n(i))_{n\ge0}$ is bounded  as well, and the claim of the lemma follows. 
\end{proof}


\begin{thebibliography}{23}
\providecommand{\natexlab}[1]{#1}
\providecommand{\url}[1]{\texttt{#1}}
\expandafter\ifx\csname urlstyle\endcsname\relax
  \providecommand{\doi}[1]{doi: #1}\else
  \providecommand{\doi}{doi: \begingroup \urlstyle{rm}\Url}\fi

\bibitem[Athreya et~al.(2008)Athreya, Ghosh, and
  Sethuraman]{athreya:ghosh:sethuraman:2008}
K.~B. Athreya, A.~P. Ghosh, and S.~Sethuraman.
\newblock Growth of preferential attachment random graphs via continuous-time
  branching processes.
\newblock \emph{Proceedings Mathematical Sciences}, 118:\penalty0 473--494,
  2008.

\bibitem[Barab\'asi and Albert(1999)]{barabasi:albert:1999}
A.~L. Barab\'asi and R.~Albert.
\newblock Emergence of scaling in random networks.
\newblock \emph{Science}, 286:\penalty0 509--512, 1999.

\bibitem[Berger et~al.(2014)Berger, Borgs, Chayes, and
  Saberi]{berger:borgs:chayes:saberi:2014}
N.~Berger, C.~Borgs, J.~T. Chayes, and A.~Saberi.
\newblock Asymptotic behavior and distributional limits of preferential
  attachment graphs.
\newblock \emph{The Annals of Probability}, 42:\penalty0 1--40, 2014.

\bibitem[Bollob\'as et~al.(2001)Bollob\'as, Riordan, Spencer, and
  Tusn\'ady]{bollobas:riordan:spencer:tusnady:2001}
B.~Bollob\'as, O.~Riordan, J.~Spencer, and G.~Tusn\'ady.
\newblock The degree sequence of a scale-free random graph process.
\newblock \emph{Random structures \& algorithms}, 18:\penalty0 279--290, 2001.

\bibitem[Bollob\'as et~al.(2003)Bollob\'as, Borgs, Chayes, and
  Riordan]{bollobas:borgs:chayes:riordan:2003}
B.~Bollob\'as, C.~Borgs, J.~Chayes, and O.~Riordan.
\newblock Directed scale-free graphs.
\newblock \emph{SODA'03: Proceedings of the fourteenth annual ACM-SIAM
  symposium on Discrete algorithms}, pages 132--139, 2003.

\bibitem[Courtneya and Bianconi(2017)]{courtneya:bianconi:2017}
O.~T. Courtneya and G.~Bianconi.
\newblock Weighted growing simplicial complexes.
\newblock \emph{Physical review. E}, 95:\penalty0 062301, 2017.

\bibitem[Engen(1978)]{engen:1978}
S.~Engen.
\newblock \emph{Stochastic Abundance Models}.
\newblock Chapman and Hall, London, 1978.

\bibitem[Fountoulakis et~al.(2022)Fountoulakis, Iyer, Mailler, and
  Sulzbach]{fountoulakis:iyer:mailler:sulzbach:2022}
N.~Fountoulakis, T.~Iyer, C.~Mailler, and H.~Sulzbach.
\newblock Dynamical models for random simplicial complexes.
\newblock \emph{The Annals of Applied Probability}, 32:\penalty0 2860--2913,
  2022.

\bibitem[Garavaglia and Stegehuis(2019)]{garavaglia:stegehuis:2019}
A.~Garavaglia and C.~Stegehuis.
\newblock Subgraphs in preferential attachment models.
\newblock \emph{Advances in Applied Probability}, 51:\penalty0 898–926, 2019.

\bibitem[Hoshino(2005)]{hoshino:2005}
N.~Hoshino.
\newblock Engen's extended negative binomial model revisited.
\newblock \emph{Annals of the Institute of Statistical Mathematics},
  57:\penalty0 369--387, 2005.

\bibitem[John(1971)]{john:1971}
F.~John.
\newblock \emph{Partial Differential Equations (Applied Mathematical Sciences
  1)}.
\newblock Springer, New York, 1971.

\bibitem[Rabenstein(1972)]{rabenstein:1972}
A.~L. Rabenstein.
\newblock \emph{Introduction to Ordinary Differential Equations: Second
  Enlarged Edition with Applications}.
\newblock Academic Press, New York and London, 1972.

\bibitem[Resnick(1992)]{resnick:1992}
S.~Resnick.
\newblock \emph{Adventures in Stochastic Processes}.
\newblock Birkh\"auser, Boston, 1992.

\bibitem[Resnick(2007)]{resnick:2007}
S.~Resnick.
\newblock \emph{Heavy-Tail Phenomena: Probabilistic and Statistical Modeling}.
\newblock Springer, New York, 2007.

\bibitem[Rivkind et~al.(2020)Rivkind, Schreier, Brenner, and
  Barak]{rivkind:schreier:brenner:barak:2020}
A.~Rivkind, H.~Schreier, N.~Brenner, and O.~Barak.
\newblock Scale free topology as an effective feedback system.
\newblock \emph{PLoS Computational Biology}, 16:\penalty0 1007825, 2020.

\bibitem[Ross(2013)]{ross:2013}
N.~Ross.
\newblock Power laws in preferential attachment graphs and stein's method for
  the negative binomial distribution.
\newblock \emph{Advances in Applied Probability}, 45:\penalty0 876--893, 2013.

\bibitem[Samorodnitsky et~al.(2016)Samorodnitsky, Resnick, Towsley, Davis,
  Willis, and Wan]{samorodnitsky:resnick:towsley:davis:willis:wan:2016}
G.~Samorodnitsky, S.~Resnick, D.~Towsley, R.~Davis, A.~Willis, and P.~Wan.
\newblock Nonstandard regular variation of in-degree and out-degree in the
  preferential attachment model.
\newblock \emph{Journal of Applied Probability}, 53:\penalty0 146--161, 2016.

\bibitem[Siu et~al.(2024)Siu, Samorodnitsky, Yu, and
  He]{siu:samorodnitsky:yu:he:2024}
C.~Siu, G.~Samorodnitsky, C.~L. Yu, and R.~He.
\newblock The asymptotics of the expected {B}etti numbers of preferential
  attachment clique complexes.
\newblock {arXiv:2305.11259}, 2024.

\bibitem[van~der Hofstad(2017)]{hofstad:2017}
R.~van~der Hofstad.
\newblock \emph{Random Graphs and Complex Networks, Volume 1}.
\newblock Cambridge University Press, Cambridge, 2017.

\bibitem[Wang and Resnick(2020)]{wang:resnick:2020}
T.~Wang and S.~Resnick.
\newblock Degree growth rates and index estimation in a directed preferential
  attachment model.
\newblock \emph{Stochastic Processes and their Applications}, 130:\penalty0
  878--906, 2020.

\bibitem[Wang and Resnick(2017)]{wang:resnick:2017}
T.~Wang and S.~I. Resnick.
\newblock Asymptotic normality of in- and out-degree counts in a preferential
  attachment model.
\newblock \emph{Stochastic models}, 33:\penalty0 229--255, 2017.

\bibitem[Wang and Resnick(2019)]{wang:resnick:2019}
T.~Wang and S.~I. Resnick.
\newblock Consistency of {H}ill estimators in a linear preferential attachment
  model.
\newblock \emph{Extremes}, 22:\penalty0 1--28, 2019.

\bibitem[Wang and Resnick(2022)]{wang:resnick:2022}
T.~Wang and S.~I. Resnick.
\newblock Asymptotic dependence of in- and out-degrees in a preferential
  attachment model with reciprocity.
\newblock \emph{Extremes}, 25:\penalty0 417--450, 2022.

\end{thebibliography}

\end{document}